\documentclass[10pt]{article}
\usepackage{amssymb}
\usepackage{amsmath}
\usepackage{theorem}
\usepackage{epsfig}
\usepackage{verbatim}
\usepackage{graphicx}
\usepackage{enumerate}
\usepackage{enumitem}
\usepackage{cancel}
\usepackage{mathtools}
\usepackage{hyperref}
\usepackage{datetime}
\usepackage{tikz}

\usepackage{color}

\newlength{\hchng}
\newlength{\vchng}

\newcommand {\bc} {\begin{center}}
\newcommand {\ec} {\end{center}}

\theoremstyle{plain}

\newtheorem{thm}{Theorem}[section]
\newtheorem{lem}{Lemma}[section]
\newtheorem{defn}{Definition}[section]
\newtheorem{prop}{Proposition}[section]
\newtheorem{cor}{Corollary}[section]	
\newtheorem{rem}{Remark}[section]

\newtheorem{pro}{Property}

\allowdisplaybreaks[2]

\newenvironment{proof}[1]{\begin{trivlist} \item[] {\em Proof of #1:}}{\hfill $\Box$
                      \end{trivlist}}

\newcommand{\ud}{\,\mathrm{d}}
\newcommand{\la}{\lambda}

\newcommand{\R}{\mathbb{R}}

\newcommand{\Om}{\Omega}

\newcommand{\pa}{\partial}
\newcommand{\eps}{\epsilon}

\newcommand{\norm}[1]{\left\lVert#1\right\rVert}

\newcommand{\Addresses}{{% additional braces for segregating \footnotesize
  \bigskip
  \footnotesize

  T.~Beck, \textsc{Department of Mathematics, University of North Carolina,
    Chapel Hill, North Carolina}\par\nopagebreak
  \textit{E-mail address}: \texttt{tdbeck@email.unc.edu}

}}

\title{The torsion function of convex domains of high eccentricity}

\date{}     
\author{Thomas Beck}
\date{\today}                                           % Activate to display a given date or no date

\begin{document}
\maketitle

\begin{abstract}

The torsion function of a convex planar domain $\Omega$ has convex level sets, but explicit formulae are known only for rectangles and ellipses. Here we study the torsion function on convex planar domains of high eccentricity. We obtain an approximation for the torsion function by viewing the domain as a perturbation of a rectangle in order to define an approximate Green's function for the Laplacian. For a class of convex domains we use this approximation to establish sharp bounds on the Hessian and the infinitesimal shape of the level sets around its maximum. We also use these results to compare the behaviour of the torsion function and the first eigenfunction of the Dirichlet Laplacian around their respective maxima. 

\end{abstract}

\section{Introduction}

The torsion function $v(x,y)$ satisfies
\begin{eqnarray*}
    \left\{ \begin{array}{rlc}
    \Delta v(x,y) & = -1 & \text{in } \Om \\
   v(x,y) & = 0& \text{on } \pa \Om.
    \end{array} \right.
\end{eqnarray*}
Throughout, $\Omega$ will be a convex planar domain, and in this case, Makar-Limanov \cite{ML}  shows that $v^{1/2}$ is concave and so $v$ has convex level sets. One of the main aims of this paper is to study the behaviour of $v$ near its maximum, with estimates that are uniform as the eccentricity of $\Omega$ increases. We will do this by looking at the second derivatives of $v$ near the maximum, as by Taylor's theorem they govern the infinitesimal shape of the level sets around the maximum. Denoting $C_{\Omega}$ to be the infinite cylinder with cross-section $\Omega$, of constant density, the integral of $v(x,y)$ is a measure of the resistance of $C_\Omega$ to a twist about the $z$-axis (torsion). The torsion function $v(x,y)$ itself is also equal to the expected first exit time from $\Omega$ of Brownian motion started at the point $(x,y)$. Therefore, the maximum of $v$ gives the point in $\Omega$ where the exit time is maximized, and the shape of the level sets around the maximum determine how the expected exit time decreases as we move away from the maximum. To study the second derivatives of $v$, and of independent interest, we also establish an approximation of the torsion function for domains of high eccentricity, by viewing the domain as a perturbation of a rectangle or ellipse, where we can write down explicit formulae. 
\\
\\
By rotating $\Omega$ so that its projection onto the $y$-axis is the smallest among any direction, and dilating, we can ensure that it is of the following form:  $\Omega$ can be written as 
\begin{align*}
\Omega = \{(x,y) \in \R^2: x\in[a,b], f_1(x) \leq y \leq f_2(x)\},
\end{align*}
for functions $f_1(x)$, $f_2(x)$ with $0 \leq f_1(x) \leq f_2(x) \leq 1$, which are convex and concave respectively. The height function $h(x) = f_2(x) - f_1(x)$ is concave, and satisfies
%\begin{figure} \label{fig:torsion1}
%\includegraphics[width=1.0\textwidth]{torsion1.png}\vspace{-0.2in}
%\caption{The normalisation of the domain $\Omega$}
%\end{figure}
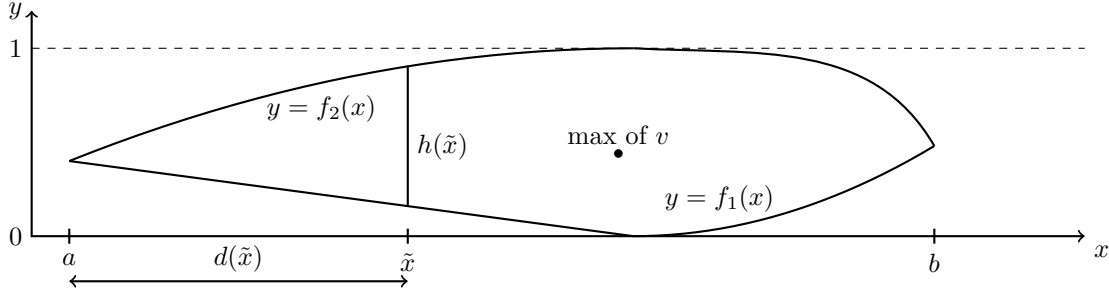
\begin{figure} \label{fig:torsion1}
\begin{tikzpicture}
\draw [thick, <->] (1,3) -- (1,0) -- (15,0);
\draw [thick, <->] (1.5,-0.6) -- (6,-0.6);
\draw [thick] (6,0.4) -- (6,2.27);
\draw [dashed] (1,2.5) -- (15,2.5);
\draw[thick] (1.5,1) -- (9,0);
\draw[thick, domain=9:13] plot (\x, {1.2*(\x-9)^2/16});
\draw[thick, domain=1.5:9] plot (\x, {2.5 - 1.5*(\x-9)^2/(7.5)^2});
\draw[thick] (9,2.5) to [out=357,in=1200] (13,1.2);
\draw [thick] (1.5,-0.1) node [below] {$a$} -- (1.5,0.1);
\draw [thick] (6,-0.1) node [below] {$\tilde{x}$}  -- (6,0.1);
\draw [thick] (13,-0.1) node [below] {$b$}  -- (13,0.1);
\node[above] at (3.75,-0.6) {$d(\tilde{x})$};
\node[above left] at (11,0.2) {$y=f_1(x)$};
\node[below right] at (4,2) {$y=f_2(x)$};
\node [below right] at (15,0) {$x$};
\node[right] at (6,1.2) {$h(\tilde{x})$};
\node [left] at (1,3) {$y$};
\node [left] at (1,2.5) {$1$};
\node [left] at (1,0) {$0$};
\draw [fill] (8.8,1.1) circle [radius=.05];;
\node [above] at (8.8,1.1) {max of $v$};
\end{tikzpicture}
\caption{The normalisation of the domain $\Omega$}
\end{figure}
\begin{align*}
0 \leq h(x) \leq 1, \qquad \max_{x\in[a,b]}h(x) = 1. 
\end{align*}
See Figure \ref{fig:torsion1} for an example  of such a domain $\Omega$. The domain $\Omega$ has inner radius comparable to $1$, and diameter comparable to $N=b-a$. By the maximum principle $v\geq0$ in $\Omega$, and the maximum of $v$ is comparable to $1$.

\subsection{An approximation for the torsion function}

 When $\Omega$ is a rectangle or an ellipse, we have an explicit formula for $v$: For $\Omega = \left[-\tfrac{1}{2}N,\tfrac{1}{2}N\right] \times [0,1]$, the torsion function is given by 
\begin{align} \label{eqn:torsion-rect}
\tfrac{1}{2}y\left(1 - y\right) -  \frac{2}{\pi^3}\sum_{n\geq 1}\frac{1-(-1)^n}{n^3\cosh(\tfrac{1}{2}n \pi N)} \sin \left(n\pi y\right) \cosh\left(n\pi x\right),
\end{align}
while when $\Omega$ is the ellipse of major axis $\tfrac{1}{2}N$, minor axis $\tfrac{1}{2}$, centred at the origin, the torsion function is
\begin{align} \label{eqn:torsion-ell}
\tfrac{1}{8}\left(\tfrac{1}{N^2} + 1\right)^{-1}\left(1 - \tfrac{4}{N^2}x^2 - 4y^2\right).
\end{align}
In \eqref{eqn:torsion-rect}, when $x$ is away from the boundary of the interval $\left[-\tfrac{1}{2}N,\tfrac{1}{2}N\right]$, we can think of $\tfrac{1}{2}y\left(1 - y\right)$ as being the main term as $N$ increases, in the sense that the infinite sum can be bounded by $e^{-\pi d(x)}$. Here $d(x)$ is the distance of $x$ from the boundary of the interval. For the ellipse we can write the domain as
\begin{align*}
\left\{(x,y)\in\mathbb{R}^2: -\tfrac{1}{2}N\leq x\leq \tfrac{1}{2}N, - \sqrt{\tfrac{1}{4} - \tfrac{1}{N^2}x^2}\leq y \leq \sqrt{\tfrac{1}{4} - \tfrac{1}{N^2}x^2}\right\}
\end{align*}
and the torsion function as
\begin{align*}
\tfrac{1}{2}\left(\tfrac{1}{N^2} + 1\right)^{-1}\left(y +  \sqrt{\tfrac{1}{4} - \tfrac{1}{N^2}x^2}\right) \left(\sqrt{\tfrac{1}{4} - \tfrac{1}{N^2}x^2}-y\right). 
\end{align*}
Again for $N$ large this has the main term $\tfrac{1}{2}\left(y +  \sqrt{\tfrac{1}{4} - \tfrac{1}{N^2}x^2}\right) \left(\sqrt{\tfrac{1}{4} - \tfrac{1}{N^2}x^2}-y\right)$. To study the torsion function for general convex planar domains $\Omega$, we will view $\Omega$ as a perturbation of a rectangular domain. As the diameter of $\Omega$ increases, we will consider the approximation of the torsion function by  
\begin{align} \label{eqn:v1}
v_1(x,y) = \frac{1}{2}(y-f_1(x))(f_2(x)-y) . 
\end{align}
Note that in the rectangular and ellipse case, $v_1(x,y)$ is precisely the term picked out when $N$ is large. Our first main theorem studies the extent to which $v(x,y)$ is approximated by $v_1(x,y)$, with a bound that becomes stronger as the diameter of $\Omega$ and $d(x)$ increases.
\begin{thm} \label{thm:Approx}
Let $\tilde{x}\in[a,b]$ be given, with  $h(\tilde{x}) \geq \tfrac{1}{2}\max_{x\in[a,b]}h(x)=\tfrac{1}{2}$. Setting $d(\tilde{x}) = \min\{\tilde{x}-a,b-\tilde{x}\}$, given $c^*>0$, there exist constants $c_1$, $C_1$ depending only on $c^*$ such that
\begin{align*}
\left|v(\tilde{x},y) - v_1(\tilde{x},y)\right| \leq C_1e^{-c_1d(\tilde{x})} + C_1 \sup_{|x-\tilde{x}|\leq \tfrac{3}{4}d(\tilde{x})}e^{-c_1|x-\tilde{x}|}\left|h(x)-h(\tilde{x})\right|, \\
\left|\pa_{x}v(\tilde{x},y) \right| \leq C_1e^{-c_1d(\tilde{x})} + C_1 \sup_{|x-\tilde{x}|\leq \tfrac{3}{4}d(\tilde{x})}e^{-c_1|x-\tilde{x}|}\left|h(x)-h(\tilde{x})\right|,
\end{align*}
for all $y\in [f_1(\tilde{x})+c^*,f_2(\tilde{x})-c^*]$.
\end{thm}
\begin{rem} \label{rem:constant}
The constants $c_1$, $C_1$ are in particular independent of the domain $\Omega$ itself $($and so are uniform in the diameter of $\Omega$$)$. Throughout, we will describe a constant as an absolute constant if it can be chosen universally $($independent of $\Omega$$)$, and otherwise will state which other constants it depends on. 
\end{rem}
To prove Theorem \ref{thm:Approx}, we will view the domain $\Omega$ as being a perturbation of an appropriately chosen rectangle. In particular, we will use the exact Green's function for the Laplacian of the rectangle, to define an approximate Green's function for the portion of $\Omega$ near $\tilde{x}$, and then use this to derive an expression for $v-v_1$ and $\pa_xv$. We will give the precise definition in Definition \ref{defn:approx-Greens} when we prove Theorem \ref{thm:Approx} in Section \ref{sec:Approx}.

\subsection{The Hessian of $v$ at its maximum}

We will use Theorem \ref{thm:Approx} to study the behaviour of $v(x,y)$ near its maximum and near the thickest part of the domain $\Omega$, around a point $\bar{x}$ such that $h(\bar{x}) = 1$. In Theorem 1 in \cite{SS}, Steinerberger shows that the level sets of the torsion function near its maximum may have eccentricity that is exponential in the diameter of $\Omega$ but no larger. In fact, this is sharp, based on the form of the torsion function for the rectangle from \eqref{eqn:torsion-rect}. Steinerberger also obtains an estimate on the Hessian at the maximum in terms of the maximum and minimum of the curvature of the boundary of the domain (see Proposition 1 in \cite{SS}). This estimate uses maximum principle techniques inspired by the work of Payne and Philippin, \cite{PP}, following on from the work of Makar-Limanov, \cite{ML}, where it is shown that $v^{1/2}$ is concave. However, this estimate is not sharp, as for example it does not recover the second derivative estimates for the torsion function for the ellipse given in \eqref{eqn:torsion-ell}. In Theorem \ref{thm:second-max} we obtain comparable upper and lower bounds on the Hessian of the torsion function in terms of the shape of the domain $\Omega$ around the point $\bar{x}$ such that $h(\bar{x})=1$. These bounds hold provided the height function $h(x)$ decays away from its maximum of $1$ in a certain uniform way (see Property \ref{pro:max} and Remark \ref{rem:pro1}). 
\\
\\
 Let $(x^*,y^*)$ be the point where $v$ attains its maximum, with $v(x^*,y^*) = v^*>0$. Using the notation of Theorem \ref{thm:Approx}, we write
\begin{align*}
\text{Error}(\tilde{x}) =  C_1e^{-c_1d(\tilde{x})} + C_1 \sup_{|x-\tilde{x}|\leq \tfrac{3}{4}d(\tilde{x})}e^{-c_1|x-\tilde{x}|}\left|h(x)-h(\tilde{x})\right|.
\end{align*}
Let $M>2$ be given. 
\begin{pro} \label{pro:max}
We say that $\Omega$ has this property if there exists $\delta \in [0,1-\tfrac{1}{2}\min\{1-h(\bar{x}+M),1-h(\bar{x}-M)\}]$  such that the following holds: For each $\tilde{x}$ with $x_{-}\leq \tilde{x} \leq x_{+}$,
\begin{align*}
\emph{Error}(\tilde{x}) \leq \tfrac{1}{100}\delta.
\end{align*}
Here $x_{-},x_{+}$ are points in $[\bar{x}-M,\bar{x}]$, $[\bar{x},\bar{x}+M]$ respectively, with $h(x_{\pm}) = 1-2\delta$.
\end{pro}
Under the assumption that Property \ref{pro:max} holds, we can obtain sharp upper and lower bounds on the second derivatives of $v(x,y)$ at its maximum:
\begin{thm} \label{thm:second-max} Suppose that \emph{Property \ref{pro:max}} holds for some $M$ and for a value of $\delta$ with $\delta = \delta(M)>0$ sufficiently small. For each unit direction $n = (a,b)$, with $a^2+b^2=1$, define $\alpha_{n}$ by
\begin{align*}
\alpha_{n}  = \max\{|b|^2,\delta\}.
\end{align*}
Then, there exist constants $c_1^* = c_1^*(M)$, $C^*_1 = C^*_1(M)$ such that
\begin{align*}
\frac{1}{C_1^*}\alpha_{n}\leq -\pa_{\nu}^2 v(x,y) \leq C_1^*\alpha_{n}
\end{align*}
for all $(x,y)\in B_{c_1^*}(x^*,y^*)$.
\end{thm}
In particular, the torsion function is concave in a neighbourhood of its maximum whenever Property \ref{pro:max} holds in this way.
\begin{rem} \label{rem:pro1}
Before continuing let us describe a class of domains for which \emph{Property \ref{pro:max}} holds: Let constants $\alpha_1$, $\beta_1$, $\gamma_1$, $\delta_1$ and $C^*$ be given with
\begin{align*}
\alpha_1>0, \quad \frac{1}{C^*}\leq \beta_1 \leq C^*, \quad \gamma_1\geq1,\quad \delta_1>0.
\end{align*}
Suppose that $h(x)$ satisfies
\begin{align*}
\left|h(x) - 1 + \frac{\beta_1|x|^{\gamma_1}}{N^{\gamma_1}}\right| \leq C^*\frac{|x|^{\gamma_1+\delta_1}}{N^{\gamma_1+\delta_1}}
\end{align*}
for $|x|\leq N^{\alpha_1}$. Then, for $N$ sufficiently large $($depending on  $\alpha_1$, $\beta_1$, $\gamma_1$, $\delta_1$, $C^*$ and $c_1$, $C_1$$)$, \emph{Property \ref{pro:max}} holds for some $M$ and $\delta$: To see this, we first note that given $M$, for all $|\tilde{x}|\leq M$, we have
\begin{align*}
\max_{x}\left||x|^{\gamma_1}-|\tilde{x}|^{\gamma_1}\right|e^{-c_1|x-\tilde{x}|} \leq A_1M^{\gamma_1-1}
\end{align*}
for a constant $A_1$ depending only on $c_1$ and $\gamma_1$. Therefore, we have
\begin{align} \label{eqn:rem-pro1}
\emph{Error}(\tilde{x}) \leq C_1e^{-\tfrac{1}{2}c_1N^{\alpha_1}} + C_1\beta_1A_1M^{\gamma_1-1} N^{-\gamma_1} + 2C_1 C^*M^{\gamma_1+\delta_1}N^{-\gamma_1-\delta_1}. 
\end{align}
At $x = M$, we have $\left|h(x) -1 - \beta_1M^{\gamma_1}N^{-\gamma_1}\right|\leq C^*M^{\gamma_1+\delta_1}N^{-\gamma_1-\delta_1}$. Therefore, by first choosing $M$ sufficiently large so that
\begin{align*}
C_1\beta_1A_1M^{\gamma_1-1} \leq \tfrac{1}{1000}\beta_1M^{\gamma_1},
\end{align*}
and then $N$ sufficiently large depending on $M$, so that the first and third terms on the right hand side of \eqref{eqn:rem-pro1} are smaller than the second, we find that Property \ref{pro:max} holds with $\delta = \tfrac{1}{2}\beta_1M^{\gamma_1}N^{-\gamma_1}$. 
\end{rem}
\begin{cor} \label{cor:second-max}
In the case where $h(x) = \tilde{h}(N^{-1}x)$, for a $C^{2,\alpha}$-smooth function $\tilde{h}$, with $\tilde{h}(0) = 1$, $\tilde{h}'(0) = 0$, $\tilde{h}''(0) <0$, Theorem \ref{thm:second-max} holds for $N$ sufficiently large, with $\delta$ comparable to $N^{-2}$. In particular, $-\pa_{x}^2v(x,y)$ is comparable to $N^{-2}$, which agrees with the bounds for the exact ellipse from \eqref{eqn:torsion-ell}. 
\end{cor}
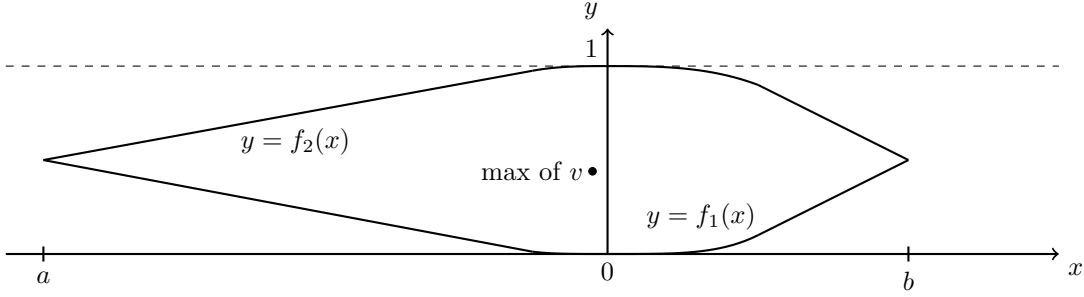
\begin{figure}
\begin{tikzpicture}
\draw [thick, ->] (1,0) -- (15,0);
\draw [thick, ->] (9,0) -- (9,3);
%\draw [thick, <->] (1.5,-0.6) -- (6,-0.6);
%\draw [thick] (6,0.4) -- (6,2.27);
\draw [dashed] (1,2.5) -- (15,2.5);
%\draw[thick] (1.5,1) -- (9,0);
\draw[thick, domain=9:11] plot (\x, {2.5 -(\x-9)^3/32});
\draw[thick, domain=8:9] plot (\x, {2.5 -(9-\x)^3/16});
\draw[thick, domain=9:11] plot (\x, {(\x-9)^4/64});
\draw[thick, domain=8:9] plot (\x, {(9-\x)^4/32});
\draw[thick] (11,0.25) -- (13,1.25);
\draw[thick] (11,2.25) -- (13,1.25);
\draw[thick] (1.5,1.25) -- (8,1/32);
\draw[thick] (1.5,1.25) -- (8,2.5-1/16);
%\draw[thick, domain=9:13] plot (\x, {1.2*(\x-9)^2/16});
%\draw[thick, domain=1.5:9] plot (\x, {2.5 - 1.5*(\x-9)^2/(7.5)^2});
%\draw[thick] (9,2.5) to [out=357,in=1200] (13,1.2);
\draw [thick] (1.5,-0.1) node [below] {$a$} -- (1.5,0.1);
%\draw [thick] (6,-0.1) node [below] {$\tilde{x}$}  -- (6,0.1);
\draw [thick] (13,-0.1) node [below] {$b$}  -- (13,0.1);
%\node[above] at (3.75,-0.6) {$d(\tilde{x})$};
\node[above left] at (11.1,0.2) {$y=f_1(x)$};
\node[below right] at (4,1.8) {$y=f_2(x)$};
\node [below right] at (15,0) {$x$};
%\node[right] at (6,1.2) {$h(\tilde{x})$};
\node [above left] at (9,3) {$y$};
\node [above left] at (9,2.5) {$1$};
\node [below] at (9,0) {$0$};
\draw [fill] (8.8,1.1) circle [radius=.05];;
\node [left] at (8.8,1.1) {max of $v$};
\end{tikzpicture}
\caption{An example of a domain satisfying Property \ref{pro:max}}  \label{fig:torsion3}
\end{figure}
A domain which satisfies Property \ref{pro:max} for $N$ sufficiently large is given in Figure \ref{fig:torsion3}. The domain in this figure corresponds to $h(x) = 1 - \tfrac{1}{10}N^{-3}|x|^3 + O(|x|^4)$ for $|x|\leq \tfrac{1}{10}N$. We will prove Theorem \ref{thm:second-max} in Section \ref{sec:second-max}. In the proof, Property \ref{pro:max} will be used together with Theorem \ref{thm:Approx} to determine the shape of a level set of $v(x,y)$ that extends precisely a distance comparable to $M$ from $x^*$ in the $x$ direction. 

\subsection{Comparison with the first Dirichlet eigenfunction}

Before proving Theorems \ref{thm:Approx} and \ref{thm:second-max}, in the next section we will use them to compare the behaviour of $v(x,y)$ near its maximum with that of the first Dirichlet eigenfunction of $\Omega$. Let $u(x,y)$ be the first Dirichlet eigenfunction of the Laplacian of $\Omega$, normalised so that $u>0$ in the interior of $\Omega$, and attains a maximum of $1$. The function $u(x,y)$ therefore satisfies
\begin{eqnarray*}
    \left\{ \begin{array}{rlc}
    \Delta u(x,y) & = -\la u(x,y) & \text{in } \Om \\
   u(x,y) & = 0& \text{on } \pa \Om.
    \end{array} \right.
\end{eqnarray*}
where $\lambda$ is the corresponding eigenvalue.  The torsion function has been used in \cite{ADJMF}, \cite{FM} and \cite{SS1} as a landscape function for predicting where high energy eigenfunctions of the Laplacian localize. Their approaches start with the inequality from \cite{FM}, which states that
 \begin{align*}
\left|u(x,y)\right| \leq \la v(x,y)\norm{u}_{L^{\infty}(\Omega)}.
 \end{align*}
 This inequality in fact holds for any eigenfunction of the Laplacian (or more generally Schr\"odinger operator with non-negative potential) on a bounded $\Omega\subset \R^{n}$. It implies that the eigenfunction $u(x,y)$ can only localize in those regions where $v(x,y)\geq c\la^{-1}$, and in particular that $v(x,y)\geq \la^{-1}$ at the maximum of $u$. In \cite{RS}, Rachh and Steinerberger also obtain  a lower bound on the torsion function at the maximum of the first eigenfunction and ask what would be the optimal constant in this lower bound (see Corollary 2 in Section 1.4 of \cite{RS}). In Proposition \ref{prop:cons2}, we show that as the diameter of the convex domain tends to infinity, the torsion function approaches its own maximal value in a rectangle around the maximum of the eigenfunction. In this convex setting of high eccentricity, this shows that the decay of the torsion function away from its maximum provides control on the location of the maximum of the eigenfunction and its surrounding level sets. The converse of this is not true - there exist families of convex domains, with diameter tending to infinity such that the eigenfunction is strictly bounded away from its maximal value at the maximum of the torsion function (see Proposition \ref{prop:cons1}). 
\\
\\
In \cite{CD} and \cite{CDK}, Cima and Derrick, and Cima, Derrick, and Kalachev presented numerical evidence suggesting that the maximum of the torsion function and first Dirichlet eigenfunction should be attained at the same point (and more generally, that this should hold true for $\Delta w = -f(w)$ with $f(w)>0$ for $w>0$). This conjecture has been disproved by Benson, Laugesen, Minion, and Siudeja in \cite{BLMS} for the semi-disk and isosceles triangle, although in their examples, the maxima are very close together compared to the diameter of the domain. In Section 1.4 in \cite{RS}, Rachh and Steinerberger give an example of a non-convex planar domain, where the maxima of the torsion function and eigenfunction are separated by $0.2$ of the diameter of the domain $\Omega$. In Proposition \ref{prop:cons1} below, we use Theorem \ref{thm:Approx} to construct a family of convex domains, with diameters tending to infinity, where the respective maxima are separated by an absolute constant multiplied by the diameter of the domain.

\section{The torsion function and first Dirichlet eigenfunction} \label{sec:Eigenfunction}

To compare the torsion function $v(x,y)$ with the Dirichlet eigenfunction $u(x,y)$, as well as Theorems \ref{thm:Approx} and \ref{thm:second-max}, we will also need some properties of $u(x,y)$, which are uniform as the diameter of $\Omega$ increases: By the work of Jerison \cite{Je} and Grieser and Jerison \cite{GJ}, the key length scale determining the behaviour of $u(x,y)$ is the length scale $L$ given by:
\begin{defn} \label{defn:L}
The length scale $L$ is the largest value such that $h(x) \geq 1- L^{-2}$ for all $x\in I$, where $I$ is an interval of length $L$.
\end{defn}
This value of $L$ satisfies $N^{1/3} \leq L \leq N$ (with the endpoints attained for a right triangle and rectangle respectively). Let $I'$ be the interval which is concentric with $I$ and of half the length. Note that by the concavity of $h$, we have the first derivative bounds
\begin{align} \label{eqn:GJ-h}
\left|h'(x)\right| \leq 2L^{-3} \text{ for } x\in I'.
\end{align}
Let the point in $\Omega$ where $u(x,y)$ attains its maximum be denoted by $(x_1,y_1)$.
\begin{prop}[Jerison, Grieser and Jerison \cite{Je}, \cite{GJ}] \label{prop:GJ}
There exists an absolute constant $C$ such that
\begin{align*}
x_1\in I', \qquad \left|y_1-\tfrac{1}{2}\right| \leq CL^{-3/2}.
\end{align*}
\end{prop}
In \cite{Be}, uniform estimates on $u(x,y)$ near its maximum are established:
\begin{prop}[Theorem 1.2 in \cite{Be}] \label{prop:eigenfunction-max}
Let $J_{\delta}$ be the length of the projection of the superlevel set $\{(x,y)\in\Omega: u(x,y) \geq 1-\delta\}$ onto the $x$-axis. There exist absolute constants $C$, $\delta_0>0$ such that for each $0<\delta<\tfrac{1}{2}$, we have
\begin{align*}
C^{-1}\sqrt{\delta} L \leq J_{\delta} \leq C\sqrt{\delta} L,
\end{align*}
and in $\{(x,y)\in\Omega:u(x,y)\geq 1-\delta_0\}$
\begin{align*}
C^{-1}L^{-2} \leq -\pa_{x}^2u(x,y) \leq CL^{-2}.
\end{align*}
\end{prop}
This shows that the first Dirichlet eigenfunction of a convex domain is strictly concave in a neighbourhood of its maximum, while is only log-concave throughout $\Omega$ (\cite{BL}). This also demonstrates a striking difference between the behaviour of the first eigenfunction and the torsion function: From Proposition \ref{prop:eigenfunction-max}, the eccentricity of the level sets  of $u$ are always bounded between $N^{1/3}$ and $N$. This is in contrast to the case of the torsion function of the rectangle, where from \eqref{eqn:torsion-rect} we see that the eccentricity can become exponentially large in $N$. In fact as shown in \cite{SS}, this is the largest possible eccentricity for the level sets of the torsion function on a convex planar domain.  
\\
\\
We will also use a simple consequence of Theorem \ref{thm:Approx} on the maximal value of $v$.
\begin{lem} \label{lem:max-value1}
The torsion function $v(x,y)$ satisfies $v(x,y)\leq \tfrac{1}{8}$ for all $(x,y)\in\Omega$. If in addition \emph{Property \ref{pro:max}} holds for some $\delta = \delta(M)$, then $v^*$ also satisfies the lower  bound
\begin{align*} 
v^* \geq \tfrac{1}{8} - \tfrac{1}{100}\delta.
\end{align*}
\end{lem} 
\begin{proof}{Lemma \ref{lem:max-value1}}
Since $v_1(\bar{x},\tfrac{1}{2}) = \tfrac{1}{8}$, the lower bound on $v^*$ follows immediately from Theorem \ref{thm:Approx} and Property \ref{pro:max}. To obtain the upper bound we show that $0\leq v(x,y)\leq \tfrac{1}{8}$ in $\Omega$: As $\Delta v(x,y) = -1$, and $v$ vanishes on $\pa\Omega$, $v$ is non-negative by the maximum principle. The function
\begin{align*}
w(x,y) = \tfrac{1}{2}y(1-y) - v(x,y)
\end{align*}
is harmonic in $\Omega$ and non-negative on $\pa\Omega$. Therefore, again by the maximum principle, we have $w(x,y)\geq 0$, and so $v(x,y) \leq \tfrac{1}{8}$. (Note that we have not used Property \ref{pro:max} to obtain this upper bound.)
\end{proof}
We first use these propositions and Lemma \ref{lem:max-value1} together with Theorem \ref{thm:Approx} to obtain an upper bound on $v^*-v(x,y)$ around the maximum of the eigenfunction. The proposition below, which holds for any convex planar domain, can also be viewed as showing that the maximum of the eigenfunction can only occur in the part of $\Omega$ where the torsion function is close to its own maximal value. 
\begin{prop} \label{prop:cons2}
Let $\Omega$ be a convex planar domain, with $L$ as in \emph{Definition \ref{defn:L}}, and let $(x_1,y_1)$ be the point where the eigenfunction $u$ attains its maximum.  Then, there exists an  absolute constant $C$  such that 
\begin{align*}
0\leq v^* - v(x,y) \leq CL^{-2}. 
\end{align*}
for all points $(x,y)$ with $|x-x_1|\leq\tfrac{1}{5}L$, $|y-y_1|\leq L^{-1}$.
\end{prop}
\begin{rem} \label{rem:cons2}
By \emph{Proposition \ref{prop:eigenfunction-max}}, at points $(x,y)$ with $|x-x_1|= \tfrac{1}{5}L$, the eigenfunction $u(x,y)$ has decreased an absolute amount from its maximum of $1$.
\end{rem}
\begin{proof}{Proposition \ref{prop:cons2}}
From Proposition \ref{prop:GJ} we have $x_1\in I'$, and $\left|y_1-\tfrac{1}{2}\right| \leq CL^{-3/2}$, and so for points $(x,y)$ in the statement of the lemma, $x\in I$, $|y-\tfrac{1}{2}|\leq CL^{-1}$. Therefore, we have
\begin{align*}
v_1(x,y) \geq \tfrac{1}{8} - \tilde{C}_1L^{-2}.
\end{align*}
Combining this with $v^*\leq \tfrac{1}{8}$ (see Lemma \ref{lem:max-value1}), and the fact that by \eqref{eqn:GJ-h} the error from Theorem \ref{thm:Approx} at $(x,y)$ can be bounded by $\tilde{C}_2L^{-2}$ gives the required upper bound on $v^* - v(x,y)$.
\end{proof}
We will now consider specific families of convex domains, where we can use Theorems \ref{thm:Approx} and \ref{thm:second-max} to obtain examples of contrasting behaviour of the torsion function and first eigenfunction near their respective maximums. For each $N\geq2$, we define the family of domains $\Omega^{(N)}_{1}$ by
\begin{align} \label{eqn:Omega-1,N}
\Omega^{(N)}_{1} = \{(x,y)\in\mathbb{R}^2: 0 \leq x \leq N, 0 \leq y \leq h^{(N)}_1(x)\},
\end{align}
where
%\begin{figure} \label{fig:torsion2}
%\includegraphics[width=1.0\textwidth]{torsion2.png}\vspace{-0.3in}
%\caption{The domain $\Omega^{(N)}_1$ and locations of the respective maxima}
%\end{figure}
\begin{figure} 
\begin{tikzpicture}
\draw [thick, <->] (1,3) -- (1,0) -- (15,0);
\draw [dashed] (1,2.5) -- (15,2.5);
\draw [thick] (13,0) -- (13,2.2);
\draw [thick] (3.5,2.5) -- (13,2.2);
\draw [thick] (1,0) -- (3.5,2.5);
\draw [thick] (1,0) -- (13,0);
\draw [thick] (3.5,-0.1) node [below] {$N^{1/2}$} -- (3.5,0.1);
\draw [thick] (13,-0.1) node [below] {$N$}  -- (13,0.1);
\node [below right] at (15,0) {$x$};
\node [left] at (1,3) {$y$};
\node [left] at (1,2.5) {$1$};
\node [left] at (1,0) {$0$};
\node at (10,2) {$y=h_1^{(N)}(x)$};
\draw [fill] (4,1.2) circle [radius=.05];
\draw [fill] (6,1.1) circle [radius=.05];
\node [below] at (4,1.2) {$(x^{*(N)},y^{*(N)})$};
\node [right] at (6,1.1) {$(x_1^{(N)},y_1^{(N)})$};
\end{tikzpicture}
\caption{The domain $\Omega^{(N)}_1$ and locations of the respective maxima} \label{fig:torsion2}
\end{figure}
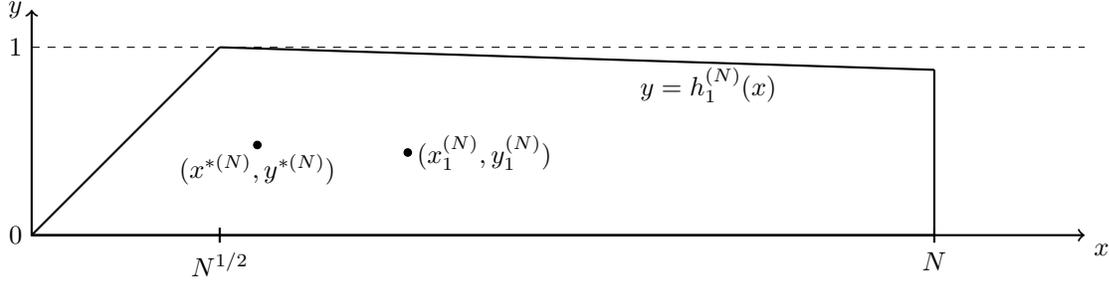
$$ h^{(N)}_1(x) = 
\begin{cases}
\frac{x}{N^{1/2}}: & 0 \leq x \leq N^{1/2}, \\
1-\frac{x-N^{1/2}}{N^3}: & N^{1/2}\leq x \leq N.
\end{cases}
$$
See Figure \ref{fig:torsion2} for the domains $\Omega_1^{(N)}$. Let $v_1^{(N)}(x,y)$ and $u_1^{(N)}(x,y)$ be the torsion function and first Dirichlet eigenfunction of $\Omega_1^{(N)}$ respectively, normalised as in the rest of the paper. Let $(x^{*(N)},y^{*(N)})$ and $(x_1^{(N)},y_1^{(N)})$ be the points where $v_1^{(N)}$ and $u_1^{(N)}$ attain their maxima. Figure 3 in \cite{RS} provides an example of a non-convex planar domain where the respective maxima are not close together, in contrast to the examples given in \cite{BLMS}. Although Proposition \ref{prop:cons2} shows that $v_1^{(N)}(x,y)$ must be close to its maximal value at $(x,y) = (x_1^{(N)},y_1^{(N)})$, this does not imply that the maxima themselves are close together.  In fact, for the family of convex domains $\Omega^{(N)}_{1}$, the values of $x^{*(N)}$ and $x_1^{(N)}$ are separated by a multiple of the diameter, uniformly as $N$ tends to infinity:
\begin{prop} \label{prop:cons1}
There exist absolute constants $C$, $c>0$ such that for all $N\geq C$, we have
\begin{align*}
\left| x^{*(N)} - x_1^{(N)} \right| \geq cN. 
\end{align*}
Moreover, the eigenfunction $u$ is uniformly bounded away from its maximal value at the maximum of $v$ in the sense that
\begin{align*}
1 - u(x^{*(N)},y^{*(N)})  \geq c .
\end{align*}
\end{prop}
\begin{proof}{Proposition \ref{prop:cons1}}
By the definition of the domains $\Omega_1^{(N)}$, there exists a constant $c_2>0$ such that $L\geq c_2N$ for all $N\geq2$, and the point $x= N^{1/2}$ is at a distance of $c_2N$ from the interval $I'$. Let the interval $I'$ be given by $[z_{1,N},z_{2,N}]$. Since by Proposition \ref{prop:GJ}, we have $x_1^{(N)} \in I'$, to establish the proposition, it is sufficient to show that for $N$ sufficiently large $x^{*(N)}\notin \tilde{I}'$, where 
\begin{align*}
\tilde{I}' = [z_{1,N}-\tfrac{1}{2}c_2N,z_{2,N}].
\end{align*}
We will show that for $x\in\tilde{I}'$, $v(x,y)$ is bounded by $v(2N^{1/2},\tfrac{1}{2})<v^*$: Using the notation of Theorem \ref{thm:Approx}, we see that
\begin{align*}
v_1(2N^{1/2},\tfrac{1}{2}) = \frac{1}{4}\left(\frac{1}{2} - N^{-5/2}\right),
\end{align*}
while there exists a constant $c_3>0$ such that $v_1(x,y) \leq \frac{1}{8} - c_3N^{-2}$  for all $(x,y)$ with $x\in\tilde{I}'$. At both $x = 2N^{1/2}$ and $x\in I'$, we can bound Error$(x)$ appearing in Theorem \ref{thm:Approx} by $CN^{-3}$. Thus, for $N$ sufficiently large, we do indeed have $v(x,y) \leq v(2N^{1/2},\tfrac{1}{2}) <v^*$  for $x\in\tilde{I}'$, as required. To obtain the lower bound on $1-u(x^{*(N)},y^{*(N)})$, we argue as follows: By Proposition \ref{prop:eigenfunction-max}, the projection of the superlevel set $\{(x,y)\in\Omega: u(x,y) \geq 1-\delta\}$ onto the $x$-axis is bounded above by $C\sqrt{\delta} N$. Therefore, by the estimate on $\left| x^{*(N)} - x_1^{(N)} \right|$, by  taking $\delta>0$ to be a sufficiently small absolute constant, we can ensure that the desired property
\begin{align*}
(x^{*(N)},y^{*(N)}) \notin \{(x,y)\in\Omega: u(x,y) \geq 1-\delta\}.
\end{align*}
holds. 
\end{proof}
From Proposition \ref{prop:eigenfunction-max}, we see that $-\pa_{x}^2u(x,y)$ is of the same order of magnitude (uniformly in $L$) in a whole superlevel set $\{(x,y)\in\Omega:u(x,y)\geq 1-\delta_0\}$, where $\delta_0>0$ is an absolute constant. This property does not necessarily hold for the torsion function. To see this, for each $N\geq2$, we define the domains $\Omega^{(N)}_{2}$ by
\begin{align} \label{eqn:Omega-2,N}
\Omega^{(N)}_{2} = \{(x,y)\in\mathbb{R}^2: |x| \leq \tfrac{1}{2}N, 0 \leq y \leq h^{(N)}_2(x)\},
\end{align}
where
$$ h^{(N)}_2(x) = 
\begin{cases}
1-\frac{|x|^2}{N^2}: &  |x| \leq N^{1/4},\\
(1- N^{-3/2})\left(1 - \frac{|x|-N^{1/4}}{\tfrac{1}{2}N-N^{1/4}} \right): & N^{1/4} \leq |x| \leq \tfrac{1}{2}N.
\end{cases}
$$
\begin{prop} \label{prop:cons3}
Let $v_2^{(N)}(x,y)$ be the torsion function of $\Omega_2^{(N)}$, normalised as in the rest of the paper, and with maximum at $(x^{*(N)},y^{*(N)})$. Then, there exist absolute constants $c$, $C$ such that for $N$ sufficiently large, we have
\begin{align*}
C^{-1}N^{-2} \leq -\pa_{x}^2v(x^{*(N)},y^{*(N)}) \leq CN^{-2},
\end{align*}
while the superlevel set $\{(x,y)\in\Omega_2^{(N)}:v(x,y) \geq v^* - cN^{-1/2}\}$ has diameter bounded above by $N^{1/2}$. This in particular ensures that
\begin{align*}
-\pa_{x}^2v(x,y) \geq 2C^{-2}cN^{-3/2}
\end{align*}
at some point $(x,y)$ in this superlevel set. 
\end{prop}
\begin{proof}{Proposition \ref{prop:cons3}}
By Remark \ref{rem:pro1} (with $\gamma_1=2$),  Property \ref{pro:max} holds for $N$ sufficiently large, with $\delta$ comparable to $N^{-2}$.  Therefore, the estimate on $-\pa_{x}^2v(x^{*(N)},y^{*(N)})$ follows immediately from Theorem \ref{thm:second-max}.  For $|x| = \tfrac{1}{2}N^{1/2}$, and $N$ sufficiently large, we have $h^{(N)}_2(x) \leq 1- \tfrac{1}{2}N^{-1/2}$, and so 
\begin{align*}
v_1(\tfrac{1}{2}N^{1/2},y) \leq \tfrac{1}{4}\left(\tfrac{1}{2} - \tfrac{1}{2}N^{-1/2}\right).  
\end{align*}
At $|x| = 2N^{1/2}$, we can bound the error appearing in Theorem \ref{thm:Approx} by $C_1N^{-1}$, 
and by Lemma \ref{lem:max-value1} we have
\begin{align*}
\tfrac{1}{8} - \tfrac{1}{100}\delta \leq v^* \leq \tfrac{1}{8}.
\end{align*}
Thus, applying Theorem \ref{thm:Approx} gives
\begin{align*}
v(\tfrac{1}{2}N^{1/2},y) \leq v^* - \tfrac{1}{10}N^{-1/2}
\end{align*}
for $N$ sufficiently large. This guarantees the upper bound on the superlevel set $\{(x,y)\in\Omega_2^{(N)}:v(x,y) \geq v^* - cN^{-1/2}\}$, and hence also the upper bound on $-\pa_{x}^2v(x,y)$ at some point in this superlevel set.
\end{proof}

\section{The approximation of the torsion function by $v_1$} \label{sec:Approx}

In this section we establish the desired estimates on $v-v_1$ and $\pa_xv$ to prove Theorem \ref{thm:Approx}, which we first restate.
\begin{thm} \label{thm:Approx1}
Let $\tilde{x}\in[a,b]$ be given, with  $h(\tilde{x}) \geq \tfrac{1}{2}\max_{x\in[a,b]}h(x)=\tfrac{1}{2}$. Setting $d(\tilde{x}) = \min\{\tilde{x}-a,b-\tilde{x}\}$, given $c^*>0$, there exist constants $c_1$, $C_1$ depending only on $c^*$ such that
\begin{align*}
\left|v(\tilde{x},y) - v_1(\tilde{x},y)\right| \leq C_1e^{-c_1d(\tilde{x})} + C_1 \sup_{|x-\tilde{x}|\leq \tfrac{3}{4}d(\tilde{x})}e^{-c_1|x-\tilde{x}|}\left|h(x)-h(\tilde{x})\right|, \\
\left|\pa_{x}v(\tilde{x},y) \right| \leq C_1e^{-c_1d(\tilde{x})} + C_1 \sup_{|x-\tilde{x}|\leq \tfrac{3}{4}d(\tilde{x})}e^{-c_1|x-\tilde{x}|}\left|h(x)-h(\tilde{x})\right|,
\end{align*}
for all $y\in [f_1(\tilde{x})+c^*,f_2(\tilde{x})-c^*]$.
\end{thm}
\begin{rem} \label{rem:Approx} 
From now on let $\tilde{x}$ and $c^*>0$ be given as in the statement of the theorem. We can also restrict to the case where $d(\tilde{x})\geq1$. We will call a quantity $g(x',y')$ an \emph{acceptable error} if there exist constants $c_1$ and $C_1$ such that
\begin{align*}
\sup_{x':|x'-\tilde{x}|\leq1} \left|g(x',y')\right| \leq C_1e^{-c_1d(\tilde{x})} + C_1 \sup_{|x-\tilde{x}|\leq \tfrac{3}{4}d(\tilde{x})}e^{-c_1|x-\tilde{x}|}\left|h(x)-h(\tilde{x})\right| \quad \emph{for } y'\in [f_1(\tilde{x})+c^*,f_2(\tilde{x})-c^*]. 
\end{align*} 
By the concavity of $h(x)$, given $c_2>0$, there exist constants $C_1 = C_1(c_2)$,  $c_1 = c_1(c_2)>0$ such that for any $x'$ with $|x'-\tilde{x}|\leq \tfrac{1}{2}d(\tilde{x})$, we have
\begin{align*}
e^{-c_2|x'-\tilde{x}|}|h'(x')| \leq C_1\sup_{|x-\tilde{x}|\leq \tfrac{3}{4}d(\tilde{x})}e^{-c_1|x-\tilde{x}|}\left|h(x)-h(\tilde{x})\right| . 
\end{align*}
\end{rem}
As we mentioned in the Introduction, we will prove Theorem \ref{thm:Approx1} by using an approximate Green's function for the Laplacian on $\Omega$ to generate an expression for $v-v_1$ and $\pa_xv$ which we can then estimate. We will define it by using the exact Green's function for a rectangle: Let $R_{c,d}$ be the rectangle given by
\begin{align*}
R_{c,d} = [0,d]\times[0,c].
\end{align*}
$R_{c,d}$ has $L^2(R_{c,d})$-normalised Dirichlet eigenfunctions
\begin{align*}
u_{n_1,n_2}(x,y) = \frac{2}{\sqrt{cd}} \sin\left(\pi \frac{n_1 x}{d}\right) \sin\left(\pi \frac{n_2 y}{c}\right),
\end{align*}
with corresponding eigenvalues $\pi^2d^{-2}n_1^2 + \pi^2c^{-2}n_2^2$, for $n_1, n_2\in\mathbb{N}$. Therefore, the Green's function for $\Delta$ on $R_{c,d}$ is given by
\begin{align} \label{eqn:tildeG}
\tilde{G}_{c,d}(x,y;x',y') = -\frac{4}{\pi^2 cd}\sum_{n_1,n_2\geq1} \frac{1}{d^{-2}n_1^2 + c^{-2}n_2^2} u_{n_1,n_2}(x,y)u_{n_1,n_2}(x',y').
\end{align}
Let $\tilde{x}\in[a,b]$ be given as in the statement of the theorem. After a translation along the $x$-axis, we set $\tilde{x} = \tfrac{1}{2}d(\tilde{x})$, and we will use $\tilde{G}_{c,d}(x,y;x',y')$ with $c= h(\tilde{x})$, $d=d(\tilde{x})$ to approximate the Green's function of $\Delta$ on $\Omega$ near $\tilde{x}$. Letting 
\begin{align} \label{eqn:Omega-x}
\Omega(\tilde{x}) = \{(x,y)\in\Omega: |x-\tilde{x}|\leq \tfrac{1}{2}d(\tilde{x})\},
\end{align}
we make the following definition:
\begin{defn} \label{defn:approx-Greens} 
For $(x,y), (x',y') \in\Omega(\tilde{x})$, define $G^{\tilde{x}}(x,y;x',y')$ by
%\begin{align} \label{eqn:G}
%G^{\tilde{x}}(x,y;x',y') = \tilde{G}_{h(\tilde{x}),d(\tilde{x})}(x,e(x,y);x',e(x',y')),
%\end{align}
%where
%\begin{align*}
%e(x,y) = \frac{(y-f_1(x))h(\tilde{x})}{h(x)}. 
%\end{align*}
\begin{align} \label{eqn:G}
G^{\tilde{x}}(x,y;x',y') = \tilde{G}_{h(\tilde{x}),d(\tilde{x})}(x,e(x,y);x',e(x',y')), \quad \emph{where} \quad  e(x,y) = \frac{(y-f_1(x))h(\tilde{x})}{h(x)}. 
\end{align}
\end{defn}
In particular, this definition ensures that $G^{\tilde{x}}(x,y;x',y') = 0$ for $y = f_1(x)$, $f_2(x)$, and $y' = f_1(x')$, $f_2(x')$. In order to use this function in our study of $v$, we first need to establish properties of $G^{\tilde{x}}(x,y;x',y')$ both near and far from the diagonal $(x,y) = (x',y')$. In Propositions \ref{prop:Green1} and \ref{prop:Green2} we show the exponential decay of $G^{\tilde{x}}(x,y;x',y')$ as $|x-x'|$ increases, as well as the nature of the singularity at $(x,y) = (x',y')$.
\begin{prop} \label{prop:Green1}
For $(x,y)$, $(x',y')\in\Omega(\tilde{x})$, the Green's function $G^{\tilde{x}}(x,y;x',y')$ can be written as
\begin{align*}
G^{\tilde{x}}(x,y;x',y') = \sum_{n\geq 1} f_n(x;x') g_n(x,y) g_n(x',y').
\end{align*}
The function $g_n(x,y)$ is given by
\begin{align*}
g_n(x,y) = \sin \left(n\pi\frac{y-f_1(x)}{h(x)}\right),
\end{align*}
and the function $f_n(x;x')$ is given by
\begin{align*}
f_n(x;x') = \frac{1}{\pi n} \sum_{m=-\infty}^{\infty}\left(\exp\left\{-2\pi \frac{d(\tilde{x})}{h(\tilde{x})}n \left|\frac{x+x'}{2 d(\tilde{x})}+m\right| \right\} - \exp\left\{-2\pi \frac{d(\tilde{x})}{h(\tilde{x})}n \left|\frac{x-x'}{2 d(\tilde{x})}+m\right| \right\}  \right) . 
\end{align*}
In particular, $f_n(x;x')$ satisfies
\begin{align*}
\pa_{x}^2f_n(x;x') = \frac{n^2\pi^2}{h(\tilde{x})^2}f_n(x;x')
\end{align*}
for $x\neq x'$, and is the Green's function for $\left(\pa_{x}^2 - \tfrac{n^2 \pi^2}{h(\tilde{x})^2}\right)$ on $\left[\tilde{x} - \tfrac{1}{2}d(\tilde{x}),\tilde{x}+\tfrac{1}{2}d(\tilde{x})\right] = [0,d(\tilde{x})]$.
\end{prop}
\begin{cor} \label{cor:Green1}
There exist absolute constants $c$, $C$ such that
\begin{align*}
\left|G^{\tilde{x}}(x,y;x',y') \right| \leq Ce^{-c|x-x'|}
\end{align*}
for all $(x,y)$, $(x',y')\in\Omega(\tilde{x})$, with $|x'-\tilde{x}|\leq1$, $|x-x'|\geq1$.
\end{cor}
To prove Propostion \ref{prop:Green1} we need the following lemma:
 \begin{lem} \label{lem:Poisson}
 For each $a>0$, $\xi\in\R$, we have the equality
 \begin{align*}
 \sum_{m =-\infty}^{\infty} \frac{1}{a^2 + 4\pi^2m^2} e^{2\pi im\xi} = \sum_{m =-\infty}^{\infty} \frac{1}{2a}e^{-a|\xi+m|}.
 \end{align*}
 \end{lem}
 \begin{proof}{Lemma \ref{lem:Poisson}}
The lemma follows immediately by applying the Poisson summation formula,
\begin{align*}
 \sum_{m =-\infty}^{\infty} \hat{f}(m)e^{2\pi im\xi} = \sum_{m =-\infty}^{\infty} f(\xi+m),
\end{align*}
to $f(x) = \frac{1}{2a}e^{-a|x|}$, with $\hat{f}(y) = \frac{1}{a^2 + 4\pi^2y^2}$.
\end{proof}
 \begin{proof}{Proposition \ref{prop:Green1}}
 For $x\neq x'$, we define $f_n(x;x')$ by
 \begin{align*}
 f_n(x;x') & = -\frac{4}{\pi^2 h(\tilde{x})d(\tilde{x})}\sum_{m\geq1} \frac{1}{d(\tilde{x})^{-2}m^2 + h(\tilde{x})^{-2}n^2} \sin \left(\frac{m\pi x}{d(\tilde{x})}\right) \sin \left(\frac{m\pi x'}{d(\tilde{x})}\right) \\
 & = \frac{4d(\tilde{x})}{h(\tilde{x})}\sum_{m=-\infty}^{\infty} \frac{1}{4\pi^2d(\tilde{x})^{2}h(\tilde{x})^{-2}n^2 + 4\pi^2m^2} \left(e^{im \pi (x+x')/d(\tilde{x})} - e^{im\pi (x-x')/d(\tilde{x})}\right) ,
  \end{align*}
so that $G^{\tilde{x}}(x,y;x',y')$ is of the desired form.  In particular, by definition $f_n(x;x')$ is the Green's function for $\left(\pa_x^2 - \tfrac{n^2 \pi^2}{h(\tilde{x})^2}\right)$ on $[0,d(\tilde{x})]$, and satisfies $\pa_{x}^2f_n(x;x') = \tfrac{n^2\pi^2}{h(\tilde{x})^2}f_n(x;x')$. Applying Lemma \ref{lem:Poisson} with $\xi = \tfrac{x\pm x'}{2 d(\tilde{x})}$ and $a = 2\pi \tfrac{d(\tilde{x})}{h(\tilde{x})}n$ therefore implies that
  \begin{align*}
  f_n(x;x') = \frac{1}{\pi n} \sum_{m=-\infty}^{\infty}\left(\exp\left\{-2\pi \frac{d(\tilde{x})}{h(\tilde{x})}n \left|\frac{x+x'}{2 d(\tilde{x})}+m\right| \right\} - \exp\left\{-2\pi \frac{d(\tilde{x})}{h(\tilde{x})}n \left|\frac{x-x'}{2 d(\tilde{x})}+m\right| \right\}  \right),
  \end{align*}
  as given in the statement of the proposition. 
\end{proof}
\begin{proof}{Corollary \ref{cor:Green1}}
For $x$, $x'$, with $|x-x'|\geq1$, $|x'-\tilde{x}|\leq1$, we can bound the sum over $m\geq1$ in the expression for $f_n(x;x')$ from Proposition \ref{prop:Green1} by
\begin{align*}
|f_n(x;x')| \leq Cn^{-1}e^{-cn|x-x'|}.
\end{align*}
Since $|g_n(x,y)|\leq 1$, we therefore have
\begin{align*}
\left|G^{\tilde{x}}(x,y;x',y')\right| \leq C\sum_{n\geq1}n^{-1}e^{-cn|x-x'|} \leq Ce^{-c|x-x'|}
\end{align*}
for $|x-x'|\geq 1$,  where $C$ is an absolute constant, changing from line-to-line.
\end{proof}
We now study the behaviour of $G^{\tilde{x}}(x,y;x',y')$ near the diagonal $(x,y) = (x',y')$. 
\begin{prop} \label{prop:Green2}
Let $1<p<\infty$ be given. Then, there exist constants $c$, $C$, depending only on $p$ such that the following statements hold:  For any $x \neq x'$, with $|x'-\tilde{x}| \leq 1$, we have the bound
\begin{align*}
\left|\tilde{G}_{h(\tilde{x}),d(\tilde{x})}(x,y;x',y')\right| \leq C\log\left(|x-x'|^{-1}\right). 
\end{align*} 
Defining the operators $T^{(0)}$, $T^{(1)}$ and $T^{(2)}$ by
 \begin{align*}
 T^{(0)}f(x',y')  & \coloneqq \int_{\Omega(\tilde{x})} \tilde{G}_{h(\tilde{x}),d(\tilde{x})}(x,y;x',y')f(x,y) \ud x \ud y, \\
  T^{(1)}f(x',y') & \coloneqq\int_{\Omega(\tilde{x})} \nabla_{x,y}\tilde{G}_{h(\tilde{x}),d(\tilde{x})}(x,y;x',y')f(x,y) \ud x \ud y, \\
  T^{(2)}f(x',y') & \coloneqq\int_{\Omega(\tilde{x})} \nabla^2_{x,y}\tilde{G}_{h(\tilde{x}),d(\tilde{x})}(x,y;x',y')f(x,y) \ud x \ud y.
 \end{align*}
 Then,
 \begin{align*}
 \left|T^{(0)}f(x',y')\right| \leq C\sup_{(x,y)\in \Omega(\tilde{x})}e^{-c|x-x'|}|f(x,y)|,
 \end{align*}
 and letting $U$ be any subset of $\Omega(x^*)$ with diameter comparable to $1$, and containing $(\tilde{x},y)$ for some $y\in[f_1(\tilde{x}),f_2(\tilde{x})]$, we have
 \begin{align*}
 \|T^{(1)}f\|_{L^p(U)} & \leq C \|e^{-c|x-\tilde{x}|}f(x,y) \|_{L^p(\Omega(\tilde{x}))} \\
 \|T^{(2)}f\|_{L^p(U)} & \leq C \|e^{-c|x-\tilde{x}|}f(x,y) \|_{L^p(\Omega(\tilde{x}))} .
 \end{align*}
\end{prop}
\begin{proof}{Proposition \ref{prop:Green2}}
Recalling the relationship between $G^{\tilde{x}}(x,y;x',y')$ and $\tilde{G}_{h(\tilde{x}),d(\tilde{x})}(x,y;x',y')$ from Definition \ref{defn:approx-Greens}, we have the same bounds as in Proposition \ref{prop:Green1} for $\tilde{G}_{h(\tilde{x}),d(\tilde{x})}(x,y;x',y')$. In particular, for $x'$ with $|x'-\tilde{x}|\leq 1$, there exists constants $c$, $C$ such that
\begin{align*}
\left|\tilde{G}_{h(\tilde{x}),d(\tilde{x})}(x,y;x',y')\right| \leq C\sum_{n\geq 1}\frac{1}{n}e^{-c n|x-x'|}.
\end{align*}
Given $\beta>0$, we have the bound
\begin{align*}
\int_{1}^{\infty}\frac{1}{t}e^{-\beta t} \ud t  \leq \int_{1}^{\beta^{-1}}\frac{1}{t} \ud t + \beta\int_{\beta^{-1}}e^{-\beta t} \ud t  = \ln\left(\beta^{-1}\right) + e^{-1}.
\end{align*}
Thus, via the integral test with $\beta = c|x-x'|$, for $x\neq x'$ we have the bound
\begin{align*}
\left|\tilde{G}_{h(\tilde{x}),d(\tilde{x})}(x,y;x',y')\right| \leq C\log\left(|x-x'|^{-1}\right). 
\end{align*}
This estimate implies that $\tilde{G}_{h(\tilde{x}),d(\tilde{x})}(x,y;x',y')$ is an integrable kernel, and combining this with the exponential decay estimate from Corollary \ref{cor:Green1}, the bound for the operator $T^{(0)}$ follows immediately.
\\
\\
To prove the estimates for $T^{(1)}$ and $T^{(2)}$ we argue as follows. By the definition of $\tilde{G}_{h(\tilde{x}),d(\tilde{x})}(x,y;x',y')$, we see that $T^{(0)}f(x,y)$ satisfies the equation
\begin{align*}
\Delta T^{(0)}f(x,y) = f(x,y),
\end{align*}
in $[0,d(\tilde{x})]\times [0,h(\tilde{x})]$, and vanishes on the boundary of the rectangle. Therefore, by elliptic regularity,  for any $1<p<\infty$ we can bound the first and second derivatives of $T^{(0)}f(x,y)$ in $L^p$ in terms of the $L^p$-norm of $T^{(0)}f(x,y)$ itself and $f(x,y)$.  Moreover, the kernel $\tilde{G}_{h(\tilde{x}),d(\tilde{x})}(x,y;x',y')$ is symmetric in $(x,y)$ and $(x',y')$, and so by duality the proposition then follows from these estimates.
\end{proof}
\begin{rem} \label{rem:Green2}
Since $h(\tilde{x})/h(x)$, $h(\tilde{x})/h(x')$ are bounded from above and below on $[0,d(\tilde{x})]$, we obtain the same bounds for the operators
  \begin{align*}
   \tilde{T}^{(0)}f(x',y') & \coloneqq\int_{\Omega(\tilde{x})} \tilde{G}_{h(\tilde{x}),d(\tilde{x})}(x,e(x,y);x',e(x',y'))f(x,y) \ud x \ud y, \\
 \tilde{T}^{(1)}f(x',y') & \coloneqq\int_{\Omega(\tilde{x})} \left(\nabla_{x,y}\tilde{G}_{h(\tilde{x}),d(\tilde{x})}\right)(x,e(x,y);x',e(x',y'))f(x,y) \ud x \ud y, \\
  \tilde{T}^{(2)}f(x',y') & \coloneqq\int_{\Omega(\tilde{x})} \left(\nabla^2_{x,y}\tilde{G}_{h(\tilde{x}),d(\tilde{x})}\right)(x,e(x,y);x',e(x',y'))f(x,y) \ud x \ud y.
 \end{align*}
\end{rem}
To obtain an expression for $v-v_1$ we will use the equation
\begin{align} \label{eqn:v1-main}
-\int_{\Omega(\tilde{x})}G^{\tilde{x}}(x,y;x',y') \ud x \ud y  = \int_{\Omega(\tilde{x})}G^{\tilde{x}}(x,y;x',y') \Delta v(x,y)\ud x \ud y .
\end{align}
We first use Propositions \ref{prop:Green1} and \ref{prop:Green2} to study the left hand side of \eqref{eqn:v1-main}.
\begin{lem} \label{lem:v1}
Let $x'$ with $|x'-\tilde{x}|\leq1$ be given. For $G^{\tilde{x}}(x,y;x',y')$ as in \eqref{eqn:G} and $v_1(x,y)$ as in \eqref{eqn:v1} we have
\begin{align*}
-\int_{\Omega(\tilde{x})}G^{\tilde{x}}(x,y;x',y') \ud x \ud y  = v_1(x',y') + \emph{Error},
\end{align*}
for an \emph{acceptable error}. 
\end{lem}
\begin{proof}{Lemma \ref{lem:v1}}
For $x'$ fixed, the Fourier series of $v_1(x',y')$ on $[f_1(x'),f_2(x')]$ is given by
\begin{align} \label{eqn:v1a}
\frac{2h(x')^2}{\pi^3}\sum_{n\geq 1}\frac{1-(-1)^n}{n^3} \sin \left(n\pi\frac{y'-f_1(x')}{h(x')}\right) . 
\end{align}
We now use Proposition \ref{prop:Green1} to approximate the integral of $G^{\tilde{x}}(x,y;x',y')$, and show that this Fourier series appears as the main term. Referring to the expression for $f_n(x;x')$ from Lemma \ref{prop:Green1}, for $x'$ satisfying $|x'-\tilde{x}|\leq1$, there exists constants $C_1$, $c_1>0$, such that
\begin{align*}
\left|f_n(x;x') +\frac{1}{\pi n}\exp\left\{-\frac{\pi n|x-x'|}{h(\tilde{x})}\right\}\right| \leq C_1e^{-c_1nd(\tilde{x})}. 
\end{align*}
Therefore, up to terms that can be included in the Error, we are left to consider
\begin{align} \label{eqn:v1b}
\int_{\Omega(\tilde{x})}\sum_{n\geq 1}\frac{1}{\pi n}\exp\left\{-\frac{\pi n|x-x'|}{h(\tilde{x})}\right\} \sin \left(n\pi\frac{y-f_1(x)}{h(x)}\right)  \sin \left(n\pi\frac{y'-f_1(x')}{h(x')}\right) \ud x \ud y.
\end{align}
Since by Proposition \ref{prop:Green2}, $G^{\tilde{x}}(x,y;x',y')$ has an integrable singularity at $(x,y) = (x',y')$, we can swap the order of summation and integration. Computing the integral in $y$, \eqref{eqn:v1b} is equal to
\begin{align} \label{eqn:v1c}
\sum_{n\geq1}\frac{1-(-1)^n}{\pi^2 n^2}\left(\int_{0}^{d(\tilde{x})}h(x)\exp\left\{-\frac{\pi n|x-x'|}{h(\tilde{x})}\right\}\ud x\right)  \sin \left(n\pi\frac{y'-f_1(x')}{h(x')}\right).
\end{align}
Adding and subtracting $h(\tilde{x})\exp\left\{-\frac{\pi n|x-x'|}{h(\tilde{x})}\right\}$ in the integrand, we see that \eqref{eqn:v1c} equals
\begin{equation}\label{eqn:v1d}
\begin{aligned} 
& 2\sum_{n\geq1}\frac{1-(-1)^n}{\pi^3 n^3}h(\tilde{x})^2 \sin \left(n\pi\frac{(y'-f_1(x'))}{h(x')}\right) \\ 
& +\sum_{n\geq1}\frac{1-(-1)^n}{\pi^2 n^2}\left(\int_{0}^{d(\tilde{x})}\left(h(x)-h(\tilde{x})\right)\exp\left\{-\frac{\pi n|x-x'|}{h(\tilde{x})}\right\}\ud x\right)  \sin \left(n\pi\frac{y'-f_1(x')}{h(x')}\right).
\end{aligned}
\end{equation}
up to boundary terms at $x = 0, d(\tilde{x})$, which can be included in the Error. Comparing this with the expression in \eqref{eqn:v1a}, we find that the first sum in \eqref{eqn:v1d} equals $v_1(x',y')$ up to an admissible error. Since $h(\tilde{x})\geq\tfrac{1}{2}$, the second sum in \eqref{eqn:v1d} can be immediately included in the error, and this completes the proof of the lemma.
\end{proof}
We now return to the right hand side of \eqref{eqn:v1-main}. We will integrate by parts to move the derivatives away from $v(x,y)$ and then combine with Lemma \ref{lem:v1} to get our expression for $v-v_1$. Given $\eps>0$, set $\Omega_{\eps}(\tilde{x}) = \Omega(\tilde{x})\cap\{(x,y):|x-x'| >\eps\}$. Since $G^{\tilde{x}}(x,y;x',y')$ has an integrable singularity, we can rewrite \eqref{eqn:v1-main} as
 \begin{align} \label{eqn:Gv1}
 -\int_{\Omega(\tilde{x})}G^{\tilde{x}}(x,y;x',y') \ud x \ud y = \lim_{\eps\to0}\int_{\Omega_{\eps}(\tilde{x})} G^{\tilde{x}}(x,y;x',y')\Delta v(x,y) \ud x \ud y,
 \end{align}
 with the left hand side as in Lemma \ref{lem:v1}. On the right hand side of \eqref{eqn:Gv1}, we integrate by parts to move the derivatives from $v(x,y)$ onto the kernel $G^{\tilde{x}}(x,y;\tilde{x},y')$. Since we have control in $L^{\infty}$ on only $h'(x)$ (and not $h''(x)$), we will do this in a way that ensures that at most one derivative is applied to $h(x)$. By Proposition \ref{prop:Green1}, the infinite sum in $G^{\tilde{x}}(x,y;\tilde{x},y')$ converges uniformly for $\eps>0$ fixed. Therefore, for the part  of \eqref{eqn:Gv1}  containing a factor of $\pa_{x}^2v(x,y)$, we integrate by parts one time in $x$ to obtain
  \begin{align}  \label{eqn:Gv2}
 & -  \lim_{\eps\to0}\int_{\Omega_{\eps}(\tilde{x})}\sum_{n \geq 1}f_n(x;x') \pa_{x} g_n(x,y)g_n(x',y') \pa_{x} v(x,y) \ud x \ud y  \\ \label{eqn:Gv3}
& -  \lim_{\eps\to0}\int_{\Omega_{\eps}(\tilde{x})}\sum_{n \geq 1}\pa_{x}f_n(x;x')  g_n(x,y)g_n(x',y') \pa_{x} v(x,y) \ud x \ud y.
 \end{align}
Since $G^{\tilde{x}}(x,y;x',y') = 0$ for $(x,y)\in\pa \Omega(\tilde{x})$, we do not get any boundary terms on $\pa \Omega(\tilde{x})$. Also, by the pointwise bounds on $G^{\tilde{x}}(x,y;x',y')$ from Proposition \ref{prop:Green2}, the boundary terms on  $|x-x'| =\eps$ vanish  as $\eps$ tends to $0$. We now integrate by parts again in \eqref{eqn:Gv3} to get the integrals
 \begin{align} \label{eqn:Gv4}
 &  \lim_{\eps\to0}\int_{\Omega_{\eps}(\tilde{x})}\sum_{n \geq 1}\pa^2_{x}f_n(x;x')  g_n(x,y)g_n(x',y')  v(x,y) \ud x \ud y \\
  & \label{eqn:Gv5}
  + \lim_{\eps\to0}\int_{\Omega_{\eps}(\tilde{x})}\sum_{n \geq 1}\pa_{x}f_n(x;x') \pa_{x} g_n(x,y)g_n(x',y')  v(x,y) \ud x \ud y
 \end{align}
 together with the following boundary terms: Since $v(x,y) = 0$ on $\pa\Omega$, the only boundary terms from $\pa\Omega(\tilde{x})$, come from those $(x,y)\in\Omega$ with $|x-\tilde{x}| = \tfrac{1}{2}d(\tilde{x})$. Thus, by Corollary \ref{cor:Green1} these terms consist of acceptable error terms. We also get the boundary term on $|x-x'|=\eps$ equal to
 \begin{align*}
  \lim_{\eps\to0}\int_{(x,y)\in \Omega_{\eps}(\tilde{x}), |x-x'|=\eps}\sum_{n \geq 1}\pa_{x}f_n(x;x')  g_n(x,y)g_n(x',y') v(x,y)  \ud y.
 \end{align*}
 Since $f_n(x;x')$ is the Green's function for $\left(\pa_x^2-\frac{n^2\pi^2}{h(\tilde{x})^2}\right)$ on  $\left[\tilde{x} - \tfrac{1}{2}d(\tilde{x}),\tilde{x}+\tfrac{1}{2}d(\tilde{x})\right]$, this integral is equal to
 \begin{align*}
 \int_{f_1(x')}^{f_2(x')}\sum_{n \geq 1}  g_n(x',y)g_n(x',y') v(x',y)  \ud y = v(x',y') .
 \end{align*}
 To obtain an expression for $v(x',y')$, we will now study the remaining part of the right hand side of \eqref{eqn:Gv1} coming from
 \begin{align*}
 \lim_{\eps\to0}\int_{\Omega_{\eps}(\tilde{x})} G^{\tilde{x}}(x,y;x',y')\pa_y^2 v(x,y) \ud x \ud y.
 \end{align*}
 Integrating by parts twice in $y$, and using $v(x,y) = G^{\tilde{x}}(x,y;x',y') = 0$ for $y=f_1(x)$, $y=f_2(x)$, this integral becomes
 \begin{align} \label{eqn:Gv6}
 -\lim_{\eps\to0}\int_{\Omega_{\eps}(\tilde{x})}\sum_{n \geq 1}\frac{n^2}{h(x)^2}f_n(x;x') g_n(x,y)g_n(x',y')  v(x,y) \ud x \ud y.
 \end{align}
 Since by Proposition \ref{prop:Green1} we have $\pa_x^2f_n(x;x') = \tfrac{n^2\pi^2}{h(\tilde{x})^2}f_n(x;x')$ for $x\neq x'$, we see that \eqref{eqn:Gv4} $+$ \eqref{eqn:Gv6} equals
 \begin{align} \label{eqn:Gv7}
  -\lim_{\eps\to0}\int_{\Omega_{\eps}(\tilde{x})}\sum_{n \geq 1}\left(\frac{1}{h(\tilde{x})^2}- \frac{1}{h(x)^2}\right)\pi^2n^2f_n(x;x') g_n(x,y)g_n(x',y')  v(x,y) \ud x \ud y.
 \end{align}
 Bringing everything together, we have established the following lemma.
 \begin{lem} \label{lem:v-expression}
Let $(x',y')\in \Omega(\tilde{x})$ with $|x'-\tilde{x}|\leq 1$. Then, we have the expression
 \begin{align*}
 v_1(x',y') = v(x',y') + \eqref{eqn:Gv2} +  \eqref{eqn:Gv5} +  \eqref{eqn:Gv7} + \emph{ Error},
 \end{align*}
 where the \emph{Error} is an \emph{acceptable error}. 
 \end{lem}
We will use the expression for $v(x',y') - v_1(x',y')$ from Lemma \ref{lem:v-expression} to obtain the desired bound on $v(\tilde{x},y) - v_1(\tilde{x},y)$ from Theorem \ref{thm:Approx}. We first use Proposition \ref{prop:Green2} (and Remark \ref{rem:Green2}) to show that, given $p$ with $1<p<\infty$, there exists an absolute constant $C_p$ such that
\begin{align} \label{eqn:v-expression1}
\norm{v - v_1}_{L^{p}(U)} \leq C_pe^{-c_1d(x^*)} + C_p \sup_{|x-x^*|\leq \tfrac{3}{4}d(x^*)}e^{-c_1|x-x^*|}\left|h(x)-h(x^*)\right|,
\end{align}
where $U$ is the rectangle $[\tilde{x}-1,\tilde{x}+1] \times[f_1(\tilde{x}) +c^*/2,f_2(\tilde{x}) - c^*/2]$. To do this, we will need a first estimate on $\nabla v(x,y)$:
\begin{lem} \label{lem:deriv}
There exists an absolute constant $C$ such that
\begin{align*}
\left|\nabla v(x,y)\right| & \leq C, \quad\emph{for all } (x,y)\in\Omega \emph{ with } |x-\tilde{x}| \leq \tfrac{1}{2}d(\tilde{x}), \\
\left|\pa_{x}v(x,y)\right| & \leq C|h'(x)|, \quad\emph{for all } (x,y)\in\pa\Omega \emph{ with } |x-\tilde{x}| \leq \tfrac{1}{2}d(\tilde{x}).
\end{align*}
\end{lem}
\begin{proof}{Lemma \ref{lem:deriv}}
Since $h(\tilde{x})\geq\tfrac{1}{2}$, by convexity we have $h(x)\geq\tfrac{1}{4}$ for $|x-\tilde{x}|\leq \tfrac{1}{2}d(\tilde{x})$. The convexity of $\Omega$ ensures that $v(x,y)$ decays linearly to the boundary, and the first bound in the statement of the lemma thus follows by elliptic estimates. Since $v(x,y) = 0$ on $\pa\Omega$, we have
\begin{align*}
v(x,f_1(x)) = 0, \quad v(x,f_2(x)) = 0.
\end{align*}
Differentiating these equations with respect to $x$ implies that
\begin{align*}
\pa_xv(x,f_1(x)) + f_1'(x)\pa_yv(x,f_1(x)) = 0, \quad \pa_xv(x,f_2(x)) + f_2'(x)\pa_yv(x,f_2(x)) = 0.
\end{align*}
Combining $|\pa_yv(x,f_i(x))| \leq C$ with $|f_i'(x)|\leq |h'(x)|$ establishes $\left|\pa_{x}v(x,y)\right| \leq C|h'(x)|$ on $\pa\Omega$. 
\end{proof}
We can now use Lemma \ref{lem:v-expression} to establish \eqref{eqn:v-expression1}: Let us first consider the integral in \eqref{eqn:Gv2}. Since $g_n(x,y) = \sin\left(n\pi \frac{y-f_1(x)}{h(x)}\right)$, we have
\begin{align*}
\pa_xg_n(x,y) = -n\pi\left(\frac{ f_1'(x)}{h(x)} + \frac{h'(x)(y-f_1(x))}{h(x)^2}\right) \cos \left(n\pi \frac{y-f_1(x)}{h(x)}\right), \quad \pa_yg_n(x,y) = \frac{n\pi}{h(x)}\cos\left(n\pi \frac{y-f_1(x)}{h(x)}\right).
\end{align*}
Therefore, we can write this integral as the $y$-component of $\tilde{T}^{(1)}f(x',y')$, where $\tilde{T}^{(1)}$ is as in Remark \ref{rem:Green2}, and $f(x,y)$ satisfies
\begin{align*}
\left|f(x,y)\right| \leq C |h'(x)||\pa_xv(x,y)|,
\end{align*}
for an absolute constant $C$. Since $|\pa_xv(x,y)|$ is bounded, the $L^{p}$ bounds on \eqref{eqn:Gv2} required for \eqref{eqn:v-expression1} follow from Proposition \ref{prop:Green2} and Remark \ref{rem:Approx}. The estimates on \eqref{eqn:Gv5} and \eqref{eqn:Gv7} follow analogously, this time using the bounds on $\tilde{T}^{(2)}f(x',y')$. Therefore, \eqref{eqn:v-expression1} holds. To go from $L^{p}$-estimates on $v-v_1$ to pointwise estimates, we need the following:
\begin{lem} \label{lem:Laplacian}
Let $\chi(x)\geq0$ be a smooth cut-off function, equal to $1$ for $|x-\tilde{x}| \leq 1$, and equal to $0$ for $|x-\tilde{x}|\geq 2$. Then, setting $w(x,y) = v(x,y)-v_1(x,y)$, we have
\begin{align*}
\Delta\left(\chi(x)w(x,y)\right) = \chi''(x)w(x,y) + 2\chi'(x)\pa_xw(x,y) + \chi(x)\sigma(x,y).
\end{align*}
Here $\sigma(x,y)$ is a function satisfying the bounds
\begin{align*}
\int_{|x-\tilde{x}|\leq2}\left|\sigma(x,y)\right| \ud x \leq C\sup_{|x-\tilde{x}|\leq 2} |h'(x)|
\end{align*}
for an absolute constant $C$.
\end{lem}
\begin{proof}{Lemma \ref{lem:Laplacian}}
Using $\Delta v(x,y) = -1$, we have
\begin{align*}
\Delta\left(\chi(x)w(x,y)\right) = \chi''(x)w(x,y) + 2\chi'(x)\pa_xw(x,y) + \chi(x)\sigma(x,y),
\end{align*}
with 
\begin{align*}
\sigma(x,y) = -1 -\Delta v_1(x,y) = -\frac{1}{2}y\left(f_1''(x) + f_2''(x)\right) +\frac{1}{2}\pa_{x}^2\left(f_1(x)f_2(x)\right).
\end{align*}
Since $f_1$ is convex, $f_2$ is concave, and $|f_i'(x)|\leq |h'(x)|$, the desired estimates on $\sigma(x,y)$ follow. 
\end{proof}
We do not have a pointwise bound on $\sigma(x,y)$ from Lemma \ref{lem:Laplacian}, and to overcome this we will use the following bounds on the Green's function for a subdomain of $\Omega$ near $\tilde{x}$ (see \cite{GJ1}, Lemma 6):
\begin{prop} \label{prop:Green3}
Let $G_0(x,y;x',y')$ be the Green's function for the domain
\begin{align*}
\Omega^{(0)}(\tilde{x}) = \{(x,y)\in\Omega:|x-\tilde{x}|\leq 2\}.
\end{align*}
Then, there exists an absolute constant $C_0$ such that
\begin{align*}
|G_0(x,y;x',y')| + |\pa_{x}G_0(x,y;x',y')| & \leq C_0, \quad \text{for } |x-x'|\geq\tfrac{1}{2} \\
\norm{G_0(\cdot,\cdot;x',y')}_{L^{\infty}_{x}(L^{1}_{y})} & \leq C_0.
\end{align*}
% Grieser-Jerison bounds - see (38''), which actually gives a stronger statement in terms of decay to the boundary in y-direction
\end{prop}
We can now convert \eqref{eqn:v-expression1} into a pointwise estimate in order to prove the desired estimate on $v-v_1$ in Theorem \ref{thm:Approx1}: In Lemma \ref{lem:Laplacian} let us write
\begin{align*}
\Delta(\chi(x)w(x,y)) = F_1(x,y) + F_2(x,y),
\end{align*}
with $F_1(x,y) = \chi''(x)w(x,y) + 2\chi'(x)\pa_xw(x,y)$, $F_2(x,y) = \chi(x)\sigma(x,y)$. Since $\chi(x)w(x,y)$ vanishes on the boundary of $\Omega^{(0)}(\tilde{x}) = \{(x,y)\in\Omega:|x-\tilde{x}|\leq 2\}$, we have
\begin{align*}
\chi(x')w(\tilde{x},y') = \int_{\Omega^{(0)}(\tilde{x})}G_0(x,y;\tilde{x},y') \left(F_1(x,y) + F_2(x,y)\right) \ud x \ud y.
\end{align*}
Using $\chi'(x)$, $\chi''(x) = 0$ for $|x-\tilde{x}|\leq 1$, the pointwise estimates on $G_0(x,y;\tilde{x},y')$ and $\pa_xG_0(x,y;\tilde{x},y')$ for $|x-\tilde{x}|\geq \tfrac{1}{2}$, together with \eqref{eqn:v-expression1} for $p=2$ say, implies that $\int_{\Omega^{(0)}(\tilde{x})}G_0(x,y;\tilde{x},y')F_1(x,y) \ud x \ud y$ has the required bound for $v(\tilde{x},y')-v_1(\tilde{x},y')$. We also have $F_2(x,y) = \chi(x)\sigma(x,y)$, and so using the integrated bound on $\sigma(x,y)$ from Lemma \ref{lem:Laplacian} together with the $L^{\infty}_xL^{1}_y$ estimate on $G_0(\cdot,\cdot;x',y')$ from Proposition \ref{prop:Green3}, we obtain the required bound for $\int_{\Omega^{(0)}(\tilde{x})}G_0(x,y;\tilde{x},y')F_1(x,y) \ud x \ud y$.
\\
\\
This completes the proof of the estimate on $v-v_1$ in Theorem \ref{thm:Approx1}, and so we now consider $\pa_xv(x',y')$. We will again use the approximate Green's function $G^{\tilde{x}}(x,y;x',y')$ to obtain an expression for $\pa_xv(x',y')$, and then bound the resulting terms to finish the proof of the theorem. We start with the integral
\begin{align} \label{eqn:paGv1}
   \lim_{\eps\to0}\int_{\Omega_{\eps}(\tilde{x})} \sum_{n\geq 1} \pa_xf_n(x;x') g_n(x,y) g_n(x',y')  \pa^2_{x}v(x,y) \ud x \ud y.
  \end{align}
Integrating by parts to remove an $x$-derivative from $\pa_xv(x,y)$ gives the integrals
\begin{align} \label{eqn:paGv2}
& - \lim_{\eps\to0}\int_{\Omega_{\eps}(\tilde{x})} \sum_{n\geq 1} \pa^2_xf_n(x;x') g_n(x,y) g_n(x',y')  \pa_{x}v(x,y) \ud x \ud y \\  \label{eqn:paGv3}
& - \lim_{\eps\to0}\int_{\Omega_{\eps}(\tilde{x})} \sum_{n\geq 1} \pa_xf_n(x;x') \pa_{x}g_n(x,y) g_n(x',y')  \pa_{x}v(x,y) \ud x \ud y,
\end{align}
together with the following boundary terms: Since $g_n(x,y) = 0$ for $y=f_1(x)$, $y=f_2(x)$, the only boundary terms from $\pa\Omega(\tilde{x})$, come from those $(x,y)\in\Omega$ with $|x-\tilde{x}| = \tfrac{1}{2}d(\tilde{x})$. Thus, by Corollary \ref{cor:Green1} these terms consist of acceptable error terms. As before, we also get a boundary term on $|x-x'|=\eps$, and since $f_n(x;x')$ is the Green's function for $\left(\pa_x^2-\frac{n^2\pi^2}{h(\tilde{x})^2}\right)$ on  $\left[\tilde{x} - \tfrac{1}{2}d(\tilde{x}),\tilde{x}+\tfrac{1}{2}d(\tilde{x})\right]$, this boundary term is equal to
 \begin{align*}
 \int_{f_1(x')}^{f_2(x')}\sum_{n \geq 1}  g_n(x',y)g_n(x',y') \pa_xv(x',y)  \ud y = \pa_xv(x',y').
 \end{align*}
 Using $\Delta v(x,y) = -1$, we can also write \eqref{eqn:paGv1} as
 \begin{align} \label{eqn:paGv4}
  - \lim_{\eps\to0}\int_{\Omega_{\eps}(\tilde{x})} \sum_{n\geq 1} \pa_xf_n(x;x') g_n(x,y) g_n(x',y')  \left(\pa^2_{y}v(x,y) +1\right) \ud x \ud y . 
 \end{align}
We first integrate by parts in $x$ to move the $x$ derivative away from $f_n(x;x')$. Since $G^{\tilde{x}}(x,y;x',y') = 0$ on $\pa \Omega(\tilde{x})$, we do not get any boundary terms, and so the integral in \eqref{eqn:paGv4} is equal to
\begin{align} \label{eqn:paGv5}
&  \lim_{\eps\to0}\int_{\Omega_{\eps}(\tilde{x})} \sum_{n\geq 1} f_n(x;x') \pa_{x}g_n(x,y) g_n(x',y')  \left(\pa^2_{y}v(x,y) +1\right) \ud x \ud y \\ \label{eqn:paGv6}
& +  \lim_{\eps\to0}\int_{\Omega_{\eps}(\tilde{x})} \sum_{n\geq 1} f_n(x;x') g_n(x,y) g_n(x',y')  \pa_x\pa^2_{y}v(x,y) \ud x \ud y . 
\end{align}
We then integrate by parts once in $y$ in the integral in \eqref{eqn:paGv5} to get
 \begin{align} \label{eqn:paGv7} 
    \lim_{\eps\to0}\int_{\Omega_{\eps}(\tilde{x})} \sum_{n\geq 1} f_n(x;x') g_n(x',y')\left(-\pa_{x}\pa_{y}g_n(x,y) \pa_yv(x,y) +\pa_xg_n(x,y)\right) \ud x \ud y,
 \end{align}
 together with the boundary terms on $y=f_1(x), f_2(x)$ given by,
 \begin{align} \label{eqn:paGv8}
\lim_{\eps\to0}\int_{\pa \Omega_{\eps}(\tilde{x})} \sum_{n\geq 1} f_n(x;x') \pa_xg_n(x,y) g_n(x',y') \pa_yv(x,y)  \nu_y\ud \sigma(x,y),
 \end{align}
 where $\nu_y$ is the $y$-component of the unit normal to the boundary. In \eqref{eqn:paGv6}, we integrate by parts twice in $y$ to rewrite it as
 \begin{align} \label{eqn:paGv9}
- \lim_{\eps\to0}\int_{\Omega_{\eps}(\tilde{x})} \sum_{n\geq 1} \frac{n^2}{h(x)^2} f_n(x;x') g_n(x,y) g_n(x',y')  \pa_xv(x,y) \ud x \ud y, 
 \end{align}
 together with the boundary term on $y=f_1(x), f_2(x)$,
 \begin{align} \label{eqn:paGv10}
 \lim_{\eps\to0}\int_{\pa \Omega_{\eps}(\tilde{x})} \sum_{n\geq 1} f_n(x;x') \pa_yg_n(x,y) g_n(x',y') \pa_xv(x,y)  \nu_y\ud \sigma(x,y).
 \end{align}
Again by Proposition \ref{prop:Green1} we have $\pa_x^2f_n(x;x') = \tfrac{n^2\pi^2}{h(\tilde{x})^2}f_n(x;x')$, and so $-$ \eqref{eqn:paGv2} $+$ \eqref{eqn:paGv9} equals
 \begin{align} \label{eqn:paGv11}
  \lim_{\eps\to0}\int_{\Omega_{\eps}(\tilde{x})}\sum_{n \geq 1}\left(\frac{1}{h(\tilde{x})^2}- \frac{1}{h(x)^2}\right)\pi^2n^2f_n(x;x') g_n(x,y)g_n(x',y')  \pa_xv(x,y) \ud x \ud y.
 \end{align}
  Bringing everything together, we have established the following lemma.
 \begin{lem} \label{lem:pav-expression}
Let $(x',y')\in \Omega(\tilde{x})$ with $|x'-\tilde{x}|\leq 1$. Then, $\pa_{x}v(x,y)$ is equal to the integrals
 \begin{align*}
 - \eqref{eqn:paGv3} +  \eqref{eqn:paGv7} +  \eqref{eqn:paGv11} + \emph{ Error},
 \end{align*}
 together with the boundary terms
 \begin{align*}
 \eqref{eqn:paGv8} + \eqref{eqn:paGv10}, 
 \end{align*}
 where the \emph{Error} is an \emph{acceptable error}. 
 \end{lem}
As for $v-v_1$, we will use Proposition \ref{prop:Green2} (and Remark \ref{rem:Green2}) to show that given $p$, with $1<p<\infty$, there exists a constant $C_p$ such that
\begin{align} \label{eqn:v-expression2}
\norm{\pa_xv}_{L^{p}(U)} \leq C_pe^{-c_1d(\tilde{x})} + C_p \sup_{|x-\tilde{x}|\leq \tfrac{3}{4}d(\tilde{x})}e^{-c_1|x-\tilde{x}|}\left|h(x)-h(\tilde{x})\right|,
\end{align}
where $U$ is the rectangle $[\tilde{x}-1,\tilde{x}+1] \times[f_1(\tilde{x}) +c^*/2,f_2(\tilde{x}) - c^*/2]$. Since $\pa_xv$ is harmonic, establishing \eqref{eqn:v-expression2} will complete the proof of Theorem \ref{thm:Approx}. The estimates on the double integrals in \eqref{eqn:paGv3}, \eqref{eqn:paGv7}, and \eqref{eqn:paGv11} follow from Proposition \ref{prop:Green2} exactly as for $v-v_1$. To deal with the boundary integrals in \eqref{eqn:paGv8} and \eqref{eqn:paGv10}, we first note that they each contain a factor of $f_1'(x)$, $h'(x)$ or $\pa_xv(x,y)$, and by Lemma \ref{lem:deriv}, we can bound $|\pa_xv(x,y)|$ by $C|h'(x)|$ on $\pa\Omega$. Moreover, in \eqref{eqn:paGv8} and \eqref{eqn:paGv10} we only have one derivative of $\tilde{G}_{h(\tilde{x}),d(\tilde{x})}$ appearing. Since from Proposition \ref{prop:Green2} we have control on integrals involving two derivatives of $\tilde{G}_{h(\tilde{x}),d(\tilde{x})}$, the required estimates on \eqref{eqn:paGv8} and \eqref{eqn:paGv10} follow from the trace theorem for Sobolev spaces. This establishes \eqref{eqn:v-expression2} and hence completes the proof of Theorem \ref{thm:Approx1}.

\section{The behaviour of the torsion function near its maximum} \label{sec:second-max}

We now focus on the behaviour of $v(x,y)$ for $x$ near a point $\bar{x}$ such that $h(\bar{x}) = 1$. Recalling that $(x^*,y^*)$ is the point where $v$ attains its maximum, with $v(x^*,y^*) = v^*>0$, we will prove:
\begin{thm} \label{thm:second-max1} Suppose that \emph{Property \ref{pro:max}} holds for some $M$ and for a value of $\delta$ with $\delta = \delta(M)$ sufficiently small. For each unit direction $n = (a,b)$, with $a^2+b^2=1$, define $\alpha_{n}$ by
\begin{align*}
\alpha_{n}  = \max\{|b|^2,\delta\}.
\end{align*}
Then, there exist constants $c_1^* = c_1^*(M)$, $C^*_1 = C^*_1(M)$ such that
\begin{align*}
\frac{1}{C_1^*}\alpha_{n}\leq -\pa_{\nu}^2 v(x,y) \leq C_1^*\alpha_{n}
\end{align*}
for all $(x,y)\in B_{c_1^*}(x^*,y^*)$.
\end{thm}
From now on we assume that  Property \ref{pro:max} holds for some $M$ and for a value of $\delta$ with $\delta = \delta(M)$ sufficiently small, to be specified below. To begin the proof of the theorem, we first recall from Lemma \ref{lem:max-value1} that  Theorem \ref{thm:Approx} together with Property \ref{pro:max} implies the following bound on the maximal value $v^*$.
\begin{lem} \label{lem:max-value}
The maximal value of $v$ satisfies
\begin{align*} 
\tfrac{1}{8} - \tfrac{1}{100}\delta \leq v^* \leq \tfrac{1}{8}.
\end{align*}
\end{lem} 
The key place where we use Property \ref{pro:max} is that it allows us to determine the shape of a level set of $v(x,y)$ that extends precisely a distance comparable to $M$ from $x^*$ in the $x$ direction. As shown by Makar-Limanov in \cite{ML}, $v^{1/2}$ is concave in $\Omega$ and so $v$ has convex superlevel sets. In particular, by the John lemma, \cite{Jo}, we can associate an ellipse contained in each superlevel set, so that a dilation of the ellipse about its centre by an absolute constant contains the superlevel set. 
\begin{lem} \label{lem:level1}
For $\eta$ in the range $\tfrac{1}{10} \delta \leq \eta \leq \tfrac{3}{20}\delta$, we can take the John ellipse of the superlevel sets $\Omega_{\eta} = \{(x,y)\in\Omega:v(x,y) \geq v^* - \eta\}$ to have axes parallel to the coordinate axes. Let $I_{x}^{\eta}$, $I_{y}^{\eta}$ be the projections of $\Omega_{\eta}$ onto the $x$ and $y$-axis respectively, with lengths $L_x^{\eta}$, and $L_y^{\eta}$. Then, there exists an absolute constant $C_2$ such that
\begin{align*}
C_2^{-1}|x_{+}-x_{-}| \leq L_x^{\eta} \leq |x_{+}-x_{-}|, \qquad C_2^{-1}\sqrt{\eta} \leq L_y^{\eta} \leq C_2\sqrt{\eta} . 
\end{align*}
Moreover, the distance between $I_{x}^{\delta/10}$ and $\pa I_{x}^{3\delta/20}$ is bounded below by $C_2^{-1}$, and $I_y^{\eta}$ contains the point $\tfrac{1}{2}$ for this range of $\eta$. 
\end{lem}
\begin{proof}{Lemma \ref{lem:level1}}
For $x$ fixed, the function $v_1(x,y) = \tfrac{1}{2}(y-f_1(x))(f_2(x)-y)$ attains its maximum of $\tfrac{1}{8}h(x)$ at $y=\tfrac{1}{2}(f_1(x)+f_2(x))$. Therefore, for $x = x_{\pm}$, we have
\begin{align*}
v_1(x,y) \leq \tfrac{1}{8} - \tfrac{1}{4}\delta.
\end{align*}
Combining this with the estimate on $v^*$ from Lemma \ref{lem:max-value}, and the assumed bound from Property \ref{pro:max} gives the required upper bound on $L_{x}^{\eta}$. Since $h(x)$ is concave, and attains its maximum of $1$, we have $h(x)\geq 1 - \tfrac{1}{10}\delta$ on an interval of length comparable to $x_{+}-x_{-}$. Using  Property \ref{pro:max} again thus gives the lower bound on $L_{x}^{\eta}$. For fixed $x$, $v_1(x,y)$ is a quadratic function of $y$, and so the upper and lower bounds on $L_y^{\eta}$ follow easily. Moreover, the projections of these superlevel sets onto any other direction have lengths bounded between $C_2\eta$ and $C_2^{-1}|x_{+}-x_{-}|$, which ensures that the John ellipses of $\Omega_{\eta}$ can be taken with axes parallel to the coordinate axes.  At $y = \tfrac{1}{2}$, $x = \bar{x}$, we have $v_1(x,y) = \tfrac{1}{8}$, and so $I_{y}^{\eta}$ certainly contains $\tfrac{1}{2}$ for this range of $\eta$. 
\\
\\
To obtain the separation between $I_{x}^{\delta/10}$ and $\pa I_{x}^{3\delta/20}$, we argue as follows: Let $x_1$ and $x_2$ be the two points to the right of $\bar{x}$ such that $h(x)$ equals $1-\delta$ and $1-\tfrac{26}{25}\delta$. Then, using Theorem \ref{thm:Approx} and Property \ref{pro:max}, the point $x_1$ is not contained in $I_{x}^{\delta/10}$, while $x_2\in I_{x}^{3\delta/20}$. We also have the analogous points to the left of $\bar{x}$. Therefore, to conclude the proof of the lemma, we need to obtain a lower bound on $x_2-x_1$. Since $h(x)$ decreases by $\tfrac{1}{25}\delta$ on the interval $[x_1,x_2]$, and $h(x)$ is concave, if $x_2-x_1$ is bounded above by a sufficiently small absolute constant, this would contradict the assumption from Property \ref{pro:max} that Error$(x_2) \leq \tfrac{1}{100}\delta$. 
\end{proof}
We want to combine this lemma with a Harnack inequality applied to the second derivatives of $v$ in order to obtain the bounds of Theorem \ref{thm:second-max1}. However, to apply a Harnack inequality we need a quantity that is of one sign, and a priori we only know that  $v^{1/2}$ is concave in $\Omega$, so that
 \begin{align} \label{eqn:half-concave}
v\pa_{n}^2v - \tfrac{1}{2}\left(\pa_{n}v\right)^2  \leq 0. 
\end{align}
Therefore, we first need to bound $\nabla v$ near to $(x^*,y^*)$. Let $\eps^*>0$ be a small absolute constant, to be determined in Lemma \ref{lem:second-deriv1} below, depending on constants appearing in elliptic estimates and the Harnack inequality. We assume that Property \ref{pro:max} holds with $\delta = \delta(M)>0$ sufficiently small so that $\delta < \tfrac{1}{100}\eps^*$.
\begin{lem} \label{lem:x-first}
For $\eps^*>0$ given, define the rectangle $R^{\delta,\eps^*}$ by
\begin{align*}
I_{x}^{3\delta/20} \times\left[\tfrac{1}{2}-\sqrt{\eps^*},\tfrac{1}{2}+\sqrt{\eps^*}\right].
\end{align*}
There exists an absolute constant $C_3$ $($independent of $\eps^*$$)$ such that for $(x,y)\in R^{\delta,\eps^*}$, we have the first derivative bounds
\begin{align*}
|\pa_{x}v(x,y)|\leq \tfrac{1}{100}\delta, \qquad |\pa_{y}v(x,y)|\leq C_3\sqrt{\eps^*}.
\end{align*}
\end{lem}
\begin{proof}{Lemma \ref{lem:x-first}}
The bound on $\pa_{x}v(x,y)$ follows immediately from Theorem \ref{thm:Approx} and Property \ref{pro:max}. For the bounds on $\pa_yv(x,y)$, we first note that since $\nabla v(x^*,y^*) = \textbf{0}$,  interior second derivative elliptic estimates on $v$ implies that there exists a constant $c>0$ so that  the superlevel set $\Omega_{\eps} = \{(x,y)\in\Omega: v(x,y) \geq v^*-\eps\}$ contains a disc of radius $c\sqrt{\eps}$ centred at $(x^*,y^*)$. Therefore, the function
\begin{align*}
\tilde{v}(x,y) = \eps^{-1}\left(v(\sqrt{\eps}x+x^*,\sqrt{\eps}y+y^*) - v^*+\eps\right).
\end{align*}
 satisfies $\Delta \tilde{v} =-1$,  attains a maximum of $1$ at the origin, and vanishes on the boundary of a region of inner radius at least $c$. The gradient of $\tilde{v}$ is thus bounded away from the boundary of this region. Since $\delta < \tfrac{1}{100}\eps^*$, there exist absolute constants $\tilde{c}_1, \tilde{c}_2>0$ such that distance between the level set $\{(x,y)\in\Omega: v= v^*-\tilde{c}_1\eps^*\}$ and the  rectangle $R^{\delta,\eps^*}$ is bounded below by $\tilde{c}_2\sqrt{\eps^*}$. The estimate on $\pa_yv(x,y)$ then follows from this estimate on $\nabla \tilde{v}$. 
\end{proof}
\begin{rem} \label{rem:separation}
By \emph{Lemma \ref{lem:level1}}, the superlevel set $\Omega_{\delta/10}$ is contained within the rectangle $R^{\delta,\eps^*}$, and there exists an absolute constant $c_3>0$ such that the Hausdorff distance between $\Omega_{\delta/10}$ and $\pa R^{\delta,\eps^*}$ is greater than $c_3\sqrt{\eps^*}$. 
\end{rem}
We will combine the estimates from Lemma \ref{lem:x-first} with \eqref{eqn:half-concave} in order to apply the Harnack inequality.
\begin{prop}[Harnack inequality, Theorem 8.17 in \cite{GT}] \label{prop:Harnack}
Let $F\geq0$ be a harmonic function in the rectangle $R^{\delta,\eps^*}$. Then, there exist a constant $\tilde{C}_1 = \tilde{C}_1(\eps^*,M)$ and an absolute constant $\tilde{C}_2$ such that for $r\leq \tfrac{1}{10}\sqrt{\eps^*}$
\begin{align*} 
\sup_{\Omega_{\delta/10}\cup B_{\sqrt{\eps^*}/10}(x^*,y^*)} F & \leq \tilde{C}_1\inf_{\Omega_{\delta/10}\cup B_{\sqrt{\eps^*}/10}(x^*,y^*)} F,\\ 
\sup_{B_{r}(x^*,y^*)} F & \leq \tilde{C}_2\inf_{B_{r}(x^*,y^*)} F.
\end{align*}
\end{prop}
Given a unit direction $n = (a,b)$, with $a^2+b^2=1$, define 
\begin{align*}
\beta_{n} & = \tfrac{1}{100}|a|\delta + |b|C_3\sqrt{\eps^*},
\end{align*}
with $C_3$ as in Lemma \ref{lem:x-first}. Note that by Lemma \ref{lem:x-first}, $\beta_{n}$ provides an upper bound on $\left|\pa_{n}v(x,y)\right|$ in $R^{\delta,\eps^*}$. We define the function $F_{n}(x,y)$ by
\begin{align*}
F_{n}(x,y) = -\pa_{n}^2v(x,y) + 5\beta_{n}^2,
\end{align*}
which is therefore non-negative in $R^{\delta,\eps^*}$ by \eqref{eqn:half-concave}, together with the lower bound on $v^*$ from Lemma \ref{lem:max-value}.  It is also a harmonic function and so  Proposition \ref{prop:Harnack} implies that
\begin{align} \label{eqn:F-Harnack1}
\sup_{\Omega_{\delta/10}\cup B_{\sqrt{\eps^*}/10}(x^*,y^*)} F_{n} & \leq \tilde{C}_1\inf_{\Omega_{\delta/10}\cup B_{\sqrt{\eps^*}/10}(x^*,y^*)} F_{n},\\ \label{eqn:F-Harnack2}
\sup_{B_{r}(x^*,y^*)} F_{n} & \leq \tilde{C}_2\inf_{B_{r}(x^*,y^*)} F_{n}.
\end{align}
We now establish Theorem \ref{thm:second-max} for some directions $n = (a,b)$, and in the process fix the value of $\eps^*$, and how small we require $\delta$ to be. Recall that $\alpha_{n} = \max\{|b|^2,\delta\}$.
\begin{lem} \label{lem:second-deriv1}
By fixing $\eps^*>0$ sufficiently small, depending only on $C_2$, $C_3$ and $\tilde{C}_2$, and then for all  $\delta<\tfrac{1}{100}\eps^*$  sufficiently small, depending on $C_2$, $\tilde{C}_1$ and $M$, there exist constants $C_2^*=C_2^*(M)$, $C_3^*=C_3^*(M)$ such that the following bounds hold: For directions $n = (a,b)$ with $|b| \geq C_3^*\sqrt{\delta}/\sqrt{\eps^*}$, and $(x,y)\in B_{\sqrt{\eps^*}/10}(x^*,y^*)$, we have
\begin{align*}
\frac{1}{C_2^*}|b|^2 \leq -\pa_{n}^2v(x,y) \leq C_2^*|b|^2.
\end{align*}
Moreover,  we have
\begin{align*}
\frac{1}{C_2^*} \delta\leq -\pa_{x}^2v(x,y) \leq C_2^*\delta
\end{align*}
for $(x,y) \in \Omega_{\delta/10}\cup B_{\sqrt{\eps^*}/10}(x^*,y^*)$.
\end{lem}
\begin{proof}{Lemma \ref{lem:second-deriv1}}
We first establish the lemma for $n = (0,1)$, by using \eqref{eqn:F-Harnack2} with $r=\sqrt{\eps^*}/10$ and choosing $\eps^*$ sufficiently small: By Lemma \ref{lem:level1}, since $\pa_{y}v(x^*,y^*) = 0$, we must have $-\pa_{y}^2v(x^*,y) \geq C_{2}^{-1}$ for some $(x^*,y) \in B_{\sqrt{\eps^*}/10}(x^*,y^*)$. Applying \eqref{eqn:F-Harnack2} thus gives
\begin{align*}
\inf_{B_{\sqrt{\eps^*}/10}(x^*,y^*)}F_{(0,1)} \geq \tilde{C}_2^{-1}C_{2}^{-1} . 
\end{align*}
Therefore, by choosing $\eps^*$ sufficiently small depending on $C_2$, $\tilde{C}_2$ and the constant $C_3$ from Lemma \ref{lem:x-first}, we must have $-\pa_{y}^2v(x,y) \geq \tfrac{1}{2} \tilde{C}_2^{-1}C_{2}^{-1}$ in $B_{\sqrt{\eps^*}/10}(x^*,y^*)$ as desired. Lemma \ref{lem:level1} also implies that $-\pa_{y}^2v(x^*,y) \leq 2C_{2}$ for some $(x^*,y) \in B_{\sqrt{\eps^*}/10}(x^*,y^*)$, and so again applying \eqref{eqn:F-Harnack2} gives the upper bound on $-\pa_{y}^2v(x,y)$ in the ball.
\\
\\
We now use \eqref{eqn:F-Harnack1} to establish the lemma for $n = (1,0)$. By Lemma \ref{lem:level1} there exist points where $-\pa_{x}^2v(x,y^*)$ is bounded above and below by $\delta$ multiplied by constants depending only on $C_2$ and $M$. Therefore, applying \eqref{eqn:F-Harnack1} with $n = (1,0)$, and for all $\delta>0$ sufficiently small (depending only on $C_2$, $M$ and $\tilde{C}_1$), we have the desired upper and lower bounds on $-\pa_{x}^2v(x,y)$ in $\Omega_{\delta/10}\cup B_{\sqrt{\eps^*}/10}(x^*,y^*)$. 
\\
\\
For $\eps^*$ fixed and $\delta$ sufficiently small as above, let $I_{n}$ be the line segment consisting of the part of $\Omega_{\delta/10}$ passing through $(x^*,y^*)$ in the direction of $n$. Then, we can choose $C_3^* = C_3^*(M)$ so that  for directions $n = (a,b)$ with $|b| \geq C_3^*\sqrt{\delta}/\sqrt{\eps^*}$, $I_{n}$ is contained within $B_{\sqrt{\eps}/10}(x^*,y^*)$. Since by Lemma \ref{lem:level1} we have sharp upper and lower bounds on the lengths of $I_{n}$, for these directions we can therefore apply \eqref{eqn:F-Harnack2} with $r = \sqrt{\eps^*}/10$, and repeat the argument for that of $-\pa_{y}^2v(x,y)$, to get the required upper and lower bounds on $-\pa_{n}^2v(x,y)$ in this ball. 
\end{proof}
Let us now fix $\eps_0>0$, with $\eps_0<\tfrac{1}{100}\eps^*$. The value of $\eps_0$ will be given (depending only on $M$) after the following lemma:
\begin{lem} \label{lem:second-deriv2}
There exist constants $a^*_1=a_1^*(M)$ and $A^*_1 = A_1^*(M)$ $($independent of $\eps_0$$)$ such that the following holds: The superlevel set $\{(x,y)\in\Omega: v \geq v^*- a^*_1\eps_0\delta\}$ is contained in $B_{\sqrt{\eps_0}}(x^*,y^*)$. The projection of this superlevel set onto the $x$ and $y$-axes have lengths between $A^{*-1}_1\sqrt{\eps_0}$ and $2\sqrt{\eps_0}$, and, $A^{*-1}_1 \sqrt{\eps_0}\sqrt{\delta}$ and $A^{*}_1 \sqrt{\eps_0}\sqrt{\delta}$ respectively.
\end{lem}
\begin{proof}{Lemma \ref{lem:second-deriv2}}
Given $a^*>0$, consider the superlevel set $\{(x,y)\in\Omega:v\geq v^*-a^*\eps_0\delta\}$. By Lemma \ref{lem:second-deriv1}, this set contains points $(x,y^*)$ for $x$ in an interval of length comparable to  $\sqrt{a^*}\sqrt{\eps_0}$. Also, for the range of directions $n=(a,b)$ in Lemma \ref{lem:second-deriv1}, it contains an interval passing through $(x^*,y^*)$ of  length comparable to  $|b|^{-1}\sqrt{a^*}\sqrt{\eps_0}\sqrt{\delta}$, (with in all cases implicit constants depending only on $C_2^*$ from Lemma \ref{lem:second-deriv1}). Since $\{(x,y)\in\Omega:v\geq v^*-a^*\eps_0\delta\}$ is convex, this is sufficient to ensure that the projection of it onto the $x$ and $y$ axes is comparable to $\sqrt{a^*}\sqrt{\eps_0}$ and $\sqrt{a^*}\sqrt{\eps_0}\sqrt{\delta}$ respectively (with implicit constants depending only on $C_2^*$ and $C_3^*$ from Lemma \ref{lem:second-deriv1}, and the now fixed $\eps^*$). Since $C_2^*$ and $C_3^*$ only depend on $M$, we can therefore choose $a^*$ sufficiently small, depending only on $M$ so that the result of the lemma holds. 
\end{proof}
We can now complete the proof of Theorem \ref{thm:second-max1} by obtaining second derivative bounds for $n = (a,b)$ with $|b|\leq C_3^*\sqrt{\delta}/\sqrt{\eps^*}$. By Lemmas \ref{lem:second-deriv1} and \ref{lem:second-deriv2}, for $n = (a,b)$ we have the first derivative bound
\begin{align*}
\left|\pa_{n}v(x,y)\right| \leq C_2^*\left(|a|\delta + |b|\sqrt{\delta}\right) A^{*}_1 \sqrt{\eps_0}
\end{align*}
in  $B_{\sqrt{\eps_0}}(x^*,y^*)$. Moreover,  for $|b|\leq C_3^*\sqrt{\delta}/\sqrt{\eps^*}$, the superlevel set $\{(x,y)\in\Omega: v \geq v^*- a^*_1\eps_0\delta\}$ consists of an interval of length comparable to $\sqrt{\eps_0}$ (with implicit constants depending on $A_1^*(M)$). In particular, for this range of $n$, given $c^*>0$ we can choose $\eps_0$ (depending only on $c^*$ and $M$) so that $|\pa_{n}v(x,y)| \leq c^*\delta$. There must be points in $\{(x,y)\in\Omega: v \geq v^*- a^*_1\eps_0\delta\}$ where $-\pa_{n}^2v(x,y)$ is  is bounded above and below by $\delta$ multiplied by constants depending only on $a_1^*(M)$ and $A_1^*(M)$. Therefore, we choose $c^*$ (and hence $\eps_0$) sufficiently small depending on these two constants, and apply the Harnack inequality from \eqref{eqn:F-Harnack2} with $r = \sqrt{\eps_0}$ to 
\begin{align*}
F_{n} = -\pa_{n}^2v(x,y) + 5c^{*2}\delta^2.
\end{align*}
This ensures that $-\pa_{n}^2v(x,y)$ is comparable to $\delta$ in $B_{\sqrt{\eps_0}}(x^*,y^*)$ with $\eps_0$ and the implicit constants depending only on $M$ as required, and this completes the proof of Theorem \ref{thm:second-max}.

\Addresses
 
\end{document}